\documentclass[11pt]{article}
\usepackage[all]{xy}
\usepackage{amssymb,amsmath,amsthm,mathtools}
\allowdisplaybreaks[4]
\usepackage{enumitem}
\usepackage{geometry}
\usepackage{changepage}
\usepackage{microtype}
\usepackage{xcolor}
\usepackage{tikz}
\definecolor{mutationone}{HTML}{D55E00}
\definecolor{mutationtwo}{HTML}{0072B2}
\definecolor{mutationthree}{HTML}{009E73}
\definecolor{tiltingfill}{HTML}{EEF4FB}

\usepackage{hyperref}
\hypersetup{colorlinks=false,pdfborder={0 0 0}}

\hypersetup{
  pdfauthor={Jiangsheng Hu, Xin Ma, Jinbi Zhang, Tiwei Zhao},
  pdftitle={Tilting realizations of derived-equivalent centralizer matrix algebras},
  pdfsubject={Tilting realizations of derived-equivalent centralizer matrix algebras},
  pdfkeywords={tilting module, centralizer matrix algebra, derived equivalence,
    fixed-source realization, weak order, Schreier multigraph}
}

\setenumerate[1]{itemsep=1pt,partopsep=0pt,parsep=\parskip,topsep=4pt}
\setitemize[1]{itemsep=1pt,partopsep=0pt,parsep=\parskip,topsep=4pt}
\setlist[itemize]{leftmargin=34pt}
\setlist[enumerate]{leftmargin=34pt}
\numberwithin{equation}{section}

\newcommand{\EnglishTitle}{Tilting realizations of derived-equivalent centralizer matrix algebras}
\newcommand{\PaperAuthorsENname}{Jiangsheng Hu$^{1}$,
Xin Ma$^{2}$, Jinbi Zhang$^{3}$ and Tiwei Zhao$^{4,\dagger}$}

\theoremstyle{plain}
\newtheorem{theorem}{Theorem}[section]
\newtheorem{lemma}[theorem]{Lemma}
\newtheorem{proposition}[theorem]{Proposition}
\newtheorem{corollary}[theorem]{Corollary}
\theoremstyle{definition}
\newtheorem{definition}[theorem]{Definition}
\newtheorem{example}[theorem]{Example}
\theoremstyle{remark}
\newtheorem{remark}[theorem]{Remark}
\newcommand{\stt}{\mathrm{s}\tau\text{-}\mathrm{tilt}\,}
\newcommand{\tiltA}{\mathrm{tilt}_1\,}
\newcommand{\Fac}{\operatorname{Fac}}
\newcommand{\rad}{\operatorname{rad}}
\newcommand{\add}{\operatorname{add}}
\newcommand{\Irr}{\operatorname{Irr}}
\newcommand{\Weak}{\operatorname{Weak}}
\newcommand{\ann}{\operatorname{ann}}
\newcommand{\ms}[1]{\{\!\{#1\}\!\}}
\newcommand{\cO}{\mathcal O}
\newcommand{\Db}{D^{\mathrm b}}
\newcommand{\Kb}{K^{\mathrm b}}
\newcommand{\op}{\mathrm{op}}
\DeclareMathOperator{\End}{End}
\DeclareMathOperator{\Hom}{Hom}
\DeclareMathOperator{\Ext}{Ext}
\DeclareMathOperator{\Tor}{Tor}
\DeclareMathOperator{\pd}{pd}
\DeclareMathOperator{\extdim}{ext.dim}

\title{\EnglishTitle}
\author{\PaperAuthorsENname}
\date{}

\begin{document}
\maketitle
\thispagestyle{empty}

\begin{center}
\footnotesize
$^{1}$School of Mathematics, Hangzhou Normal University,
Hangzhou 311121, P. R. China\\[-0.5mm]
E-mail: \href{mailto:hujs@hznu.edu.cn}{hujs@hznu.edu.cn}

\vspace{0.2mm}
$^{2}$College of Science, Henan University of Engineering, Zhengzhou 451191, P. R. China\\[-0.5mm]
E-mail: \href{mailto:maxin@haue.edu.cn}{maxin@haue.edu.cn}

\vspace{0.2mm}
$^{3}$School of Mathematical Sciences, Anhui University, Hefei 230601, P. R. China\\[-0.5mm]
E-mail: \href{mailto:zhangjb@ahu.edu.cn}{zhangjb@ahu.edu.cn}

\vspace{0.2mm}
$^{4}$School of Artificial Intelligence, Jianghan University,
Wuhan 430056, P. R. China\\[-0.5mm]
E-mail: \href{mailto:tiweizhao@jhun.edu.cn}{tiweizhao@jhun.edu.cn}

\vspace{0.2mm}
$^{\dagger}$Corresponding author
\end{center}

\vspace{1mm}
\begin{adjustwidth}{1cm}{1cm}
\noindent\footnotesize
\textbf{Abstract.}
Let $A$ be a centralizer matrix algebra over an arbitrary field.
We prove that a centralizer matrix algebra $B$ is derived equivalent to $A$
if and only if $B\simeq\End_A(T)^{\op}$ for some tilting $A$-module $T$ of
projective dimension at most one. We classify the basic tilting $A$-modules
and identify their poset with a product of right weak orders on symmetric
groups. For a fixed primary block, mutations exchange adjacent entries of
the gap word. The isomorphism classes of the resulting opposite endomorphism
algebras are parametrized by left cosets of the stabilizer of this word.
The quotient of its labeled mutation graph by the stabilizer action is a
Schreier multigraph. We also characterize when the right weak-order
orientation descends to the nonloop edges of this quotient.
\par\vspace{1mm}
\noindent
\textbf{2020 Mathematics Subject Classification.}
Primary 16D90, 16E35, 16G10; Secondary 15A27, 05E18, 06A07.

\vspace{1mm}
\noindent
\textbf{Keywords.}
tilting module; centralizer matrix algebra; derived equivalence;
fixed-source realization; weak order; successive gaps; Schreier multigraph.
\end{adjustwidth}

\section{Introduction}\label{sec:introduction}

Let $R$ be a field, and let
$A=A_0,A_1,\ldots,A_r=B$ be finite-dimensional $R$-algebras. Suppose that,
for each $i$, there is a classical tilting left $A_{i-1}$-module $T_i$ such that
$
 A_i\simeq\End_{A_{i-1}}(T_i)^{\op}.
$
By Happel's tilting theorem (see \cite{HappelDerived}), each $T_i$ induces a
triangle equivalence
$\Db(A_{i-1})\simeq\Db(A_i)$.
Composing these equivalences gives a triangle equivalence from
$\Db(A)$ to $\Db(B)$. By Rickard's derived Morita theorem
(see \cite[Theorem~6.4]{Rickard}), there is a tilting complex $P^\bullet$ over
$A$ such that
$\End_{\Kb(A\text{-}\mathrm{proj})}(P^\bullet)^{\op}\simeq B$.
Here, the complex $P^\bullet$ need not be concentrated in degree zero.

We call the following endpoint question the \emph{fixed-source
realization problem}: does there exist a tilting $A$-module $T$ such that
$
 \End_A(T)^{\op}\simeq B
$?
This asks only whether the algebra $B$ can be realized over the fixed source
$A$; it does not ask whether the given composite equivalence itself is
induced by $T$.

We solve this problem when both $A$ and $B$ are centralizer matrix algebras.
Let $c\in M_n(R)$, and write
$$
 S_n(c,R):=\{a\in M_n(R)\mid ac=ca\}.
$$
Regard $R^n$ as an $R[x]$-module by letting $x$ act as $c$. Then
$S_n(c,R)\simeq\End_{R[x]}(R^n)$.
We call an algebra of this form a \emph{centralizer matrix algebra}.

Let $\Irr(c)$ be the set of monic irreducible factors of the minimal
polynomial. For $f\in\Irr(c)$, set
$P_c(f):=\{a\geq1\mid f^a\text{ occurs as an elementary divisor}\}$ and $\kappa_c(f):=|P_c(f)|$.
Let $V_f$ be the $f$-primary component of the $R[x]$-module $R^n$.
Under the primary decomposition, one has
$$
 {R^n\simeq\bigoplus_{f\in\Irr(c)}V_f,
 \; \text{ and }
 S_n(c,R)\simeq
 \prod_{f\in\Irr(c)}\End_{R[x]}(V_f).}
$$
We call the factor $\End_{R[x]}(V_f)$ the $f$-primary block of
$S_n(c,R)$.
If $P_c(f)=\{p_1<\cdots<p_s\}$, we define
$g(P_c(f)):=(p_1,p_2-p_1,\ldots,p_s-p_{s-1})$.
Li and Xi characterize when two centralizer matrix algebras are derived equivalent in
terms of $D$-equivalence
(see \cite[Theorem~1.1 and Definition~3.1(2)]{LiXiD}). Stable equivalences
and their homological consequences for centralizer matrix algebras have
also been studied (see \cite{LiXiStable,XiZhangStable}). For each pair of
corresponding primary blocks, the derived-equivalence criterion requires the local coefficient
algebras to be isomorphic and the gap multisets to
agree. Thus it determines which centralizer matrix algebras can occur as $B$,
but it does not produce a single tilting module over the original source.
Indeed, Li--Xi's adjacent-gap construction, combined with Hu--Xi's
$\mathcal D$-split construction, gives a chain of module-induced equivalences
whose source
changes at every step (see \cite[Lemmas~2.11--2.12]{LiXiD} and
\cite[Theorem~3.5 and Lemma~3.4]{HuXi}).
Our main result shows that, within the class of centralizer matrix algebras, the
algebra $B$ can always be realized by a tilting module over $A$.

\begin{theorem}[Theorem~\ref{thm:one-step}]\label{thm:intro-one-step}
Let $A=S_n(c,R)$ and $B=S_m(d,R)$. The following are equivalent.
\begin{enumerate}[label={\textup{(\arabic*)}}]
\item $A$ and $B$ are derived equivalent as $R$-algebras.
\item $c$ and $d$ are $D$-equivalent in the sense of Li--Xi.
\item There is a tilting left $A$-module $T$ such that
$\End_A(T)^{\op}\simeq B$.
\end{enumerate}
If $B$ is basic, then $T$ may be chosen basic. The same assertion holds
with $A$ and $B$ interchanged.
\end{theorem}

The equivalence between \textup{(1)} and \textup{(2)} is Li--Xi's
theorem. The implication
\textup{(3)}~$\Longrightarrow$~\textup{(1)} follows from Happel's tilting
theorem (see \cite{HappelDerived}).
Thus the new implication is
\textup{(1)}~$\Longrightarrow$~\textup{(3)}: every centralizer matrix algebra in
the derived equivalence class of $A$ is realized over the fixed source $A$.
A Morita-theoretic lifting argument adjusts the multiplicities of the tilting
summands to realize prescribed nonbasic algebras up to isomorphism.
Corollary~\ref{cor:closure-exhaustion} also shows that
$\End_A(T)^{\op}$ is again a centralizer matrix algebra for every tilting
$A$-module $T$. Consequently, the opposite endomorphism algebras of tilting
$A$-modules are precisely the centralizer matrix algebras derived equivalent to $A$.

For a centralizer matrix algebra $B$ derived equivalent to $A$, choose a
tilting $A$-module $T$ with $B\simeq\End_A(T)^{\op}$ as in
Theorem~\ref{thm:intro-one-step}. Two standard tilting-theoretic consequences
follow. First, $B\text{-}\mathrm{mod}$ is equivalent to the
Happel--Reiten--Smal\o\ (HRS) tilt of
$A\text{-}\mathrm{mod}$ at
$
 \bigl(\Fac T,\ker\Hom_A(T,-)\bigr).
$ Here $\Fac T$ denotes the factor modules of finite direct sums of copies
of $T$; see Corollary~\ref{cor:one-step-hrs} and \cite{BrennerButler,HRS}.
Second, since $T$ has
projective dimension at most one, the extension dimensions of $A$ and $B$,
in the sense of \cite{ZhangZhengExt}, differ by at most one; see
Corollary~\ref{cor:extension-dimension-bound}.

We first classify the basic tilting modules over the fixed source and then
determine their opposite endomorphism algebras. We write
$\tiltA A$ for the poset of isomorphism classes of basic tilting left
$A$-modules, ordered by
$T\leq T'\Leftrightarrow \Fac T\subseteq\Fac T'$.
We also write $\Weak(\Sigma_r)$ for the right weak
order on the symmetric group $\Sigma_r$, and $\ms{\cdots}$ for a multiset,
including repetitions.

The following theorem gives the fixed-source parameter space; its proof is
in Section~\ref{sec:global}.

\begin{theorem}
\label{thm:intro-poset}
For every $c\in M_n(R)$, one has
$$
 \tiltA S_n(c,R)\simeq
 \prod_{f\in\Irr(c)}\Weak(\Sigma_{\kappa_c(f)}).
$$
The abstract poset determines exactly the multiset
$\ms{\kappa_c(f)\mid f\in\Irr(c),\ \kappa_c(f)\geq2}$. Together with the
center, it determines the full multiset
$\ms{\kappa_c(f)\mid f\in\Irr(c)}$.
\end{theorem}

For a primary block, let $Q$ be the faithful projective-injective module
associated with the longest uniserial summand. Every basic tilting module
contains $Q$ as a direct summand. Hence Jasso reduction gives an isomorphism
between the tilting poset and the support $\tau$-tilting poset of an
idempotent quotient
(see \cite[Theorem~3.15 and Corollary~3.16]{Jasso}). We then use the uniform
central quotient constructed in \cite[Theorem~3.1]{HLZtau}; see also
\cite[Theorem~11]{EJR}. It follows that the block contributes a factor
isomorphic to $\Weak(\Sigma_s)$, where $s$ is the number of distinct
exponents.
This extends the saturated Auslander case of Iyama--Zhang
(see \cite[Proposition~3.10, Theorem~3.18, and
Corollary~3.19]{IyamaZhang}). It also shows precisely how the support
$\tau$-tilting factor $\Weak(\Sigma_{s+1})$ from
\cite[Theorem~1.2]{HLZtau} reduces to the tilting factor
$\Weak(\Sigma_s)$. Since corresponding primary blocks in Li--Xi's criterion
have the same value of $s$, the tilting poset is a derived invariant among
centralizer matrix algebras; see Corollary~\ref{cor:derived-invariance}.

For a primary block with $s$ distinct exponents, the $s!$ tilting vertices
are indexed by $w\in\Sigma_s$. Determining their opposite endomorphism
algebras also requires the coefficient algebra, the gap word, and the
mutation labels. The reconstruction limitations are discussed at the end of
Section~\ref{sec:global}; related questions are studied in
\cite{HappelUngerReconstruction,AiharaKase,KaseInverse}.

Let $\mathcal O$ be a finite-dimensional commutative local principal
$R$-algebra, let
$g=(g_1,\ldots,g_s)\in\mathbb Z_{>0}^s$, and assume that $\sum_i g_i$ is
the Loewy length of $\mathcal O$. Set
$$
 q_i=\sum_{j=1}^i g_j, \
 A_g=\End_{\mathcal O}\left(
       \bigoplus_{i=1}^s\mathcal O/\rad^{q_i}\mathcal O\right),
  \ \ (g\cdot w)_j=g_{w(j)}.
$$
For $1\leq i<s$, let $s_i=(i,i+1)$, and let $\ell$ denote the Coxeter
length on $\Sigma_s$.

The next theorem solves the fixed-source realization problem for one primary
block.

\begin{theorem}[Theorem~\ref{thm:fixed-primary}]
\label{thm:intro-fixed}
There are basic tilting $A_g$-modules $\{T_w\mid w\in\Sigma_s\}$,
representing all elements of $\tiltA A_g$, such that
$
 T_1=A_g$ and
 $\End_{A_g}(T_w)^{\op}\simeq A_{g\cdot w}.$
If $\ell(ws_i)=\ell(w)+1$, then $T_{ws_i}$ is the downward irreducible
mutation of $T_w$ at the summand with label $i$.
\end{theorem}

A mutation with label $i$ exchanges the gaps in positions $i$ and $i+1$.
Applying this construction to each primary block and then using Morita
lifting yields Theorem~\ref{thm:intro-one-step}.

Theorem~\ref{thm:intro-fixed} naturally leads to a uniqueness question:
when do two basic tilting $A_g$-modules have isomorphic opposite
endomorphism algebras?

For each value $a$ occurring in $g$, set
$
 I_a:=\{j\in\{1,\ldots,s\}\mid g_j=a\},
 \ \
 m_a:=|I_a|,
$
and let
$$
 H_g:=\operatorname{Stab}_{\Sigma_s}(g)
 =\{h\in\Sigma_s\mid g\cdot h=g\}
 =\prod_a\operatorname{Sym}(I_a),
$$
where $\operatorname{Sym}(I_a)$ denotes the symmetric group on $I_a$.
Thus $H_g$ permutes precisely the positions carrying equal gaps.
The left cosets $H_gw$ parametrize the distinct permuted gap words. We write
$\operatorname{Sch}(H_g\backslash\Sigma_s;s_1,\ldots,s_{s-1})$ for the
undirected labeled Schreier multigraph, retaining loops and parallel edges;
its precise definition is given in Section~\ref{sec:endomorphism-classes}.
Our final main result combines Theorem~\ref{thm:endomorphism-classes},
Proposition~\ref{prop:Schreier}, and Theorem~\ref{thm:orientation-descent}.

\begin{theorem}\label{thm:intro-endomorphism-classes}
For $u,v\in\Sigma_s$, one has
$$
 \End_{A_g}(T_u)^{\op}\simeq\End_{A_g}(T_v)^{\op}
 \quad\Leftrightarrow\quad
 H_gu=H_gv.
$$
Consequently, the quotient of the labeled mutation graph by the action
$h\cdot T_w=T_{hw}$ of $H_g$, with both vertices and edges taken modulo
this action, is
$
 \operatorname{Sch}
 \bigl(H_g\backslash\Sigma_s;s_1,\ldots,s_{s-1}\bigr).
$
The right weak-order orientation descends to the nonloop edges of this
quotient if and only if every $I_a$ is an interval in
$\{1,\ldots,s\}$.
\end{theorem}

Each isomorphism class of opposite endomorphism algebras is realized by
exactly $|H_g|=\prod_a m_a!$ pairwise nonisomorphic basic tilting
$A_g$-modules. Hence there are exactly
$s!/\prod_a m_a!$ isomorphism classes of opposite endomorphism
algebras. An $i$-labeled edge at $H_gw$ is a loop precisely when
$(g\cdot w)_i=(g\cdot w)_{i+1}$. When the orientation descends, the
partial order generated by the oriented nonloop edges is the usual
parabolic weak order on multiset permutations; its rank enumeration is given in
Remark~\ref{rem:parabolic-quotient}.

Related tilting and braid-group constructions are known in the saturated
Auslander case (see \cite{Geuenich,Sauter}). Our results apply to arbitrary positive
gap words and identify the resulting opposite endomorphism algebras.

The paper is organized as follows. In Section~\ref{sec:prelim}, we introduce
the necessary definitions and reduction results. In Section~\ref{sec:global},
we prove Theorem~\ref{thm:intro-poset}. In Section~\ref{sec:gaps}, we construct
adjacent gap exchanges and prove Theorem~\ref{thm:fixed-primary}, which
contains Theorem~\ref{thm:intro-fixed}. In Section~\ref{sec:one-step}, we prove
Theorem~\ref{thm:intro-one-step} and show that the opposite endomorphism
algebras of tilting $A$-modules are precisely the centralizer matrix algebras derived
equivalent to $A$; see Corollary~\ref{cor:closure-exhaustion}. Finally, in
Section~\ref{sec:endomorphism-classes}, we prove
Theorem~\ref{thm:intro-endomorphism-classes} and describe the resulting quotient
graph and its orientation.

{ Throughout, all algebras are finite-dimensional associative unital
$R$-algebras, and all algebra isomorphisms, Morita equivalences, and derived
equivalences are $R$-linear. Unless otherwise stated, modules are finitely
generated left modules. For an algebra $A$, we write
$A\text{-}\mathrm{mod}$ for the category of finitely generated left
$A$-modules and $\Db(A):=\Db(A\text{-}\mathrm{mod})$ for its bounded
derived category. For an $A$-module $X$, let $\add X$ denote the full
subcategory of direct summands of finite direct sums of copies of $X$.
We set $\End_A(X):=\Hom_A(X,X)$, with multiplication given by composition, and
$(-)^{\op}$ denotes the opposite algebra. For $X\in\Db(A)$, we denote its
$i$th cohomology module by $H^i(X)$. We write $\tau_A$ for the
Auslander--Reiten translation on $A\text{-}\mathrm{mod}$.
All subcategories are full, additive and closed under isomorphisms.}

\section{Preliminaries}
\label{sec:prelim}

In this section, we collect the conventions and structural results.


\subsection{Tilting modules, extension dimensions, and weak orders}

\begin{definition}[{see \cite[p.~399]{HappelRingel}}]
\label{def:tilting}
Let $A$ be an algebra. An $A$-module $T$ is a \emph{tilting module} if
\begin{enumerate}[label={\textup{(T\arabic*)}}]
\item $\pd_AT\leq1$;
\item $\Ext_A^1(T,T)=0$;
\item there is an exact sequence
$0\longrightarrow A\longrightarrow T_0\longrightarrow T_1 \longrightarrow0$
with $T_0,T_1\in\add T$.
\end{enumerate}
\end{definition}

\begin{definition}[{see \cite[Chapter~I, Section~2]{HRS}}]
\label{def:torsion-pair-hrs}
Let $A$ be an algebra. A pair $(\mathcal T,\mathcal F)$ of full
subcategories of $A\text{-}\mathrm{mod}$ is a \emph{torsion pair} if
$\Hom_A(X,Y)=0$ for all $X\in\mathcal T$ and $Y\in\mathcal F$, and every
$M\in A\text{-}\mathrm{mod}$ admits an exact sequence
$0\longrightarrow M_{\mathcal T}\longrightarrow M \longrightarrow M_{\mathcal F}\longrightarrow0$
with $M_{\mathcal T}\in\mathcal T$ and
$M_{\mathcal F}\in\mathcal F$.
Its \emph{Happel--Reiten--Smal\o\ tilt} is
the heart
$$
 \mathcal H(\mathcal T,\mathcal F)
 =\left\{X\in\Db(A)\ \middle|\
 \begin{array}{l}
 H^i(X)=0\text{ for }i\neq-1,0,\\
 H^{-1}(X)\in\mathcal F,\quad H^0(X)\in\mathcal T
 \end{array}\right\}.
$$
Equivalently,
$\mathcal H(\mathcal T,\mathcal F)
=\langle\mathcal F[1],\mathcal T\rangle_{\mathrm{ex}}$, where the
right-hand side denotes extension closure in $\Db(A)$.
\end{definition}

\begin{definition}[{see \cite[Definition~2.3]{ZhangZhengExt}}]
\label{def:extension-dimension}
For full subcategories $\mathcal X$ and $\mathcal Y$ of
$A\text{-}\mathrm{mod}$, let $\mathcal X\diamond\mathcal Y$ be the closure
under finite direct sums and direct summands of the modules $M$ admitting
an exact sequence
$0\longrightarrow X\longrightarrow M\longrightarrow Y\longrightarrow0$
with $X\in\mathcal X, Y\in\mathcal Y$.
For an $A$-module $G$, set $[G]_0=\{0\}$, $[G]_1=\add G$, and
$[G]_n=[G]_1\diamond[G]_{n-1}$ for $n\geq2$. The \emph{extension
dimension} of $A$ is
$$
 \extdim(A)=\inf\{n\geq0\mid
 A\text{-}\mathrm{mod}=[G]_{n+1}\text{ for some }G\in
 A\text{-}\mathrm{mod}\}.
$$
In particular, $\extdim(A)=0$ if and only if $A$ is
representation-finite (see \cite[Lemma~2.4(1)]{ZhangZhengExt}).
\end{definition}

\begin{definition}[{see \cite[Definitions~0.1 and~0.3]{AIR}}]
\label{def:support-tau-tilting}
Let $M$ be an $A$-module and let $P$ be a finitely generated
projective $A$-module. The pair $(M,P)$ is \emph{$\tau$-rigid} if
$\Hom_A(M,\tau_A M)=0$ and $\Hom_A(P,M)=0$.
It is a \emph{support $\tau$-tilting pair} if, in addition, the
number of isomorphism classes of indecomposable direct summands of
$M\oplus P$ equals the number of isomorphism classes of simple
$A$-modules.
{
In this case, $M$ is called a support $\tau$-tilting module.}
\end{definition}

We write $\stt A$ for the poset of isomorphism classes of basic support
$\tau$-tilting $A$-modules, ordered by
$M\leq N\Leftrightarrow\Fac M\subseteq\Fac N$.
By \cite[Proposition~2.2(b)]{AIR}, tilting modules of projective
dimension at most one are precisely the faithful support $\tau$-tilting
modules. Although \cite{AIR} is stated under its standing assumptions on the
base field, the proof of this proposition is valid over an arbitrary field.
Hence
$\tiltA A=\{T\in\stt A\mid\ann_AT=0\}$.

We write $\Kb(A\text{-}\mathrm{proj})$ for the bounded homotopy category of
finitely generated projective left $A$-modules.
\begin{definition}[{see \cite[Definition~1.5 and Section~5.1]{AIR}}]
\label{def:two-term-silting}
A complex $V\in\Kb(A\text{-}\mathrm{proj})$ is \emph{presilting} if
$\Hom_{\Kb(A\text{-}\mathrm{proj})}(V,V[i])=0$ for every $i>0$. It is
\emph{silting} if, in addition,
$\operatorname{thick}(V)=\Kb(A\text{-}\mathrm{proj})$. A complex is
\emph{two-term} if it is concentrated in degrees $-1$ and $0$.
If
$V^\bullet=(P^{-1}\to P^0)$ is minimal, its \emph{$g$-vector} is
$g(V^\bullet)=[P^0]-[P^{-1}]\in K_0(A\text{-}\mathrm{proj})$.
\end{definition}

We write $2\text{-}\mathrm{silt}\,A$ for the poset of isomorphism classes
of basic two-term silting complexes, ordered by
$$
 V\leq W\Leftrightarrow
 \Hom_{\Kb(A\text{-}\mathrm{proj})}(W,V[i])=0
 \quad\text{for every }i>0.
$$
Iyama, J\o rgensen and Yang established a bijection
between basic support $\tau$-tilting modules and basic two-term silting
complexes (see \cite[Theorem~0.2]{IyamaJorgensenYang}). With these conventions,
the bijection preserves the order (see \cite[Theorem~2.7(b) and Corollary~2.8]{DIJ}), and
irreducible mutations are the Hasse covers (see \cite[Corollary~2.34]{AIR}).

\begin{definition}[{see \cite[Section~3.1]{BjornerBrenti}}]
\label{def:right-weak-order}
Let $\Sigma_s$ be the symmetric group, let $s_i=(i,i+1)$, and let $\ell$
be the Coxeter length. The \emph{right weak order}
$\Weak(\Sigma_s)$ is determined by the covers
$w\lessdot ws_i\Leftrightarrow \ell(ws_i)=\ell(w)+1$.
\end{definition}

The left and right weak orders are isomorphic via $w\mapsto w^{-1}$.

\begin{definition}[{see \cite[Definition~3.1]{HuXi}}]
\label{def:add-split}
Let $A$ be an algebra and let $X,Y,N\in A\text{-}\mathrm{mod}$. An
$A$-module homomorphism $u:X\to N'$ with $N'\in\add(N)$ is a
\emph{left $\add(N)$-approximation} if, for every $N''\in\add(N)$, every
homomorphism $X\to N''$ factors through $u$. It is \emph{left minimal} if
every $a\in\End_A(N')$ satisfying $a\circ u=u$ is invertible. Dually, a
homomorphism $v:N'\to Y$ with $N'\in\add(N)$ is a
\emph{right $\add(N)$-approximation} if every homomorphism $N''\to Y$,
with $N''\in\add(N)$, factors through $v$; it is \emph{right minimal} if
every $a\in\End_A(N')$ satisfying $v\circ a=v$ is invertible. A short
exact sequence
$0\longrightarrow X\xrightarrow{u}N'\xrightarrow{v}Y\longrightarrow0$
is an \emph{$\add(N)$-split sequence} if $u$ is a left
$\add(N)$-approximation and $v$ is a right $\add(N)$-approximation.
\end{definition}

\subsection{Reduction tools}

Now we intoduce the following
two reductions. First, Jasso reduction at a faithful
projective-injective summand identifies the tilting slice with the support
$\tau$-tilting poset of an idempotent quotient. Second, central reduction
preserves that poset together with its summand labels and mutation
directions.

\begin{proposition}
\label{prop:faithful-slice}
Let $e\in A$ be an idempotent. Assume that $Q=Ae$ is basic, faithful,
projective, and injective. Then there is a poset isomorphism
$$
 \tiltA A
 =\{T\in\stt A\mid Q\in\add T\}
 \simeq\stt(A/AeA).
$$
\end{proposition}

\begin{proof}
By \cite[Proposition~2.2(b)]{AIR}, tilting modules are precisely the
faithful support $\tau$-tilting modules. If $T$ is tilting, condition
\textup{(T3)} gives a monomorphism $Q\to T_0$ with $T_0\in\add T$.
Since $Q$ is injective, this monomorphism splits, so $Q\in\add T$. Conversely, if
$T\in\stt A$ and $Q\in\add T$, then
$\ann_AT\subseteq\ann_AQ=0$; hence $T$ is faithful and therefore tilting.
This proves the equality of the two displayed sets. Since $Q$ is
projective, $\tau_AQ=0$. Thus the Jasso reduction subcategory associated
with $Q$ is
$J(Q)=\{X\in A\text{-}\mathrm{mod}\mid\Hom_A(Q,X)=0\}$.
The natural isomorphism $\Hom_A(Ae,X)\simeq eX$ shows that $J(Q)$
consists precisely of the modules annihilated by $AeA$. Hence inflation
along $A\to A/AeA$ identifies $(A/AeA)\text{-}\mathrm{mod}$ with $J(Q)$
as exact categories. The asserted poset isomorphism now follows from Jasso reduction
(see \cite[Theorem~3.15 and Corollary~3.16]{Jasso}).
\end{proof}

For comparison, Iyama--Zhang
(see \cite[Theorem~4.5]{IyamaZhangAG}) obtain the analogous bijection for a
1-Gorenstein algebra when $Ae$ is an additive generator of all
projective-injective modules. Here no Gorenstein assumption and no such
generator condition are imposed; instead, the chosen summand $Ae$ is assumed
directly to be faithful and injective.

The second reduction passes to a quotient by an ideal
$I\subseteq\bigl(Z(A)\cap\rad A\bigr)A$ and preserves the labeled
mutation structure.

\begin{proposition}
\label{prop:central-reduction-labels}
Let $I\subseteq\bigl(Z(A)\cap\rad A\bigr)A$ be an ideal and write
$\overline A=A/I$. Then reduction induces poset isomorphisms
$$
 2\text{-}\mathrm{silt}\,A\simeq2\text{-}\mathrm{silt}\, \overline A \text{ and }
  \ \
 \stt A\simeq\stt\overline A.
$$
After the indecomposable projectives are labeled compatibly, reduction
preserves $g$-vectors, indecomposable summands, mutation labels, and mutation
directions.
\end{proposition}

\begin{proof}
By the arbitrary-field reduction theorem
(see \cite[Proposition~2.1]{HLZtau}), which extends
\cite[Theorem~11]{EJR}, reduction induces the stated poset isomorphisms and
preserves direct sums and $g$-vectors. Choose primitive idempotents
$e_1,\ldots,e_t$ such
that $Ae_i$ represent the indecomposable projectives. Since
$I\subseteq\rad A$, their images $\overline A\overline e_i$ are precisely
the indecomposable projectives of $\overline A$. Therefore
$$
 g(V^\bullet)=\sum_i g_i[Ae_i]
 \quad\Longrightarrow\quad
 g(\overline A\otimes_AV^\bullet)
   =\sum_i g_i[\overline A\overline e_i].
$$
Thus reduction preserves the projective coordinates of every summand.
The arbitrary-field uniqueness of two-term presilting $g$-vectors
(see \cite[Theorem~6.5(a) and Corollary~6.7]{DIJ}) then shows that corresponding
indecomposable summands have the same labels. Consequently, corresponding
irreducible mutations replace summands with the same label before and after
reduction. Finally, the poset isomorphism preserves the mutation direction.
\end{proof}

\subsection{Primary blocks and structural facts}

 Since Morita
equivalences preserve projective dimension, extensions, additive closure,
and the generation order, it follows that
$\tiltA\Bigl(\prod_{j=1}^q A_j\Bigr) \simeq\prod_{j=1}^q\tiltA A_j$.

For $f\in\Irr(c)$, let $m_{f,a}$ be the multiplicity of the elementary
divisor $f^a$ and let $V_f$ be the $f$-primary component of $R^n$. Then
$$
 {
 R^n\simeq\bigoplus_{f\in\Irr(c)}V_f,\;
 V_f\simeq\bigoplus_{a\in P_c(f)}
       \bigl(R[x]/(f^a)\bigr)^{m_{f,a}},\;
 S_n(c,R)\simeq\prod_{f\in\Irr(c)}\End_{R[x]}(V_f).}
$$
We call this the \emph{primary product decomposition}. This decomposition
and the Morita reduction obtained by replacing each $V_f$ with a basic
additive generator of $\add V_f$ are established over an arbitrary field in
\cite[equations~(4.1) and~(4.2)]{HLZtau}. Thus a basic algebra in the Morita
class of $S_n(c,R)$ is
\begin{equation}\label{eq:basic-centralizer}
 S_n(c,R)^{\mathrm b}\simeq
 \prod_{f\in\Irr(c)}\Lambda_f(P_c(f)),
  \ \
 \Lambda_f(P)=\End_{R[x]}\left(
   \bigoplus_{a\in P}R[x]/(f^a)\right).
\end{equation}
For each $f$, passing to a basic additive generator of $\add V_f$ retains
one copy of each indecomposable summand and eliminates the multiplicities
of repeated elementary divisors. Consequently, these multiplicities cannot
be recovered from Morita-invariant tilting data.

The tilting poset does not detect the primary blocks with
$\kappa_c(f)=1$, whose tilting posets are singletons. We use the following
formula for the center to recover their total number.

\begin{proposition}[{see \cite[Proposition~4.3 and
Corollary~4.4]{HLZtau}}]
\label{prop:center-matrix}
Let $\mu_c$ be the minimal polynomial of $c$. Then
$$
 Z(S_n(c,R))=R[c]\simeq R[x]/(\mu_c)
 \simeq\prod_{f\in\Irr(c)}\mathcal O_c(f),
  \ \ \mathcal O_c(f)=R[x]/(f^{\max P_c(f)}).
$$
The algebras $\mathcal O_c(f)$ are precisely the local factors of the center.
\end{proposition}

For a finite set $P=\{p_1<\cdots<p_s\}$ of positive integers, define its
\emph{gap word} by
$$
 g(P)=(p_1,p_2-p_1,\ldots,p_s-p_{s-1}).
$$
We call the multiset of entries of $g(P)$ the \emph{gap multiset} of
$P$. Conversely, every tuple
$g=(g_1,\ldots,g_s)\in\mathbb Z_{>0}^s$ determines
$P(g)=\{g_1,g_1+g_2,\ldots,g_1+\cdots+g_s\}$, and
$g(P(g))=g$. We call such a tuple a \emph{positive gap word}.

\begin{definition}[{see \cite[Definition~3.1(2)]{LiXiD}}]
\label{def:D-equivalence}
Matrices $c\in M_n(R)$ and $d\in M_m(R)$ are \emph{$D$-equivalent} if there
is a bijection $\sigma:\Irr(c)\to\Irr(d)$ such that, for every $f\in\Irr(c)$,
$\mathcal O_c(f)\simeq\mathcal O_d(\sigma(f))$ as $R$-algebras,
and the gap multisets of $P_c(f)$ and $P_d(\sigma(f))$ agree.
\end{definition}

This formulation is equivalent to Li--Xi's original definition, which
indexes primary blocks by their maximal elementary divisors. Their
classification are as follows.

\begin{theorem}[{see \cite[Theorem~1.1]{LiXiD}}]
\label{thm:LiXi-criterion}
Let $c\in M_n(R)$ and $d\in M_m(R)$. Then $S_n(c,R)$ and $S_m(d,R)$ are
derived equivalent if and only if $c$ and $d$ are $D$-equivalent.
\end{theorem}

The coefficient algebra, rather than its defining polynomial, is the
invariant in Definition~\ref{def:D-equivalence}. The following elementary
lemma shows that an isomorphism of coefficient algebras carries radical
powers to the corresponding radical powers and induces an
idempotent-preserving isomorphism of the associated endomorphism algebras.

\begin{lemma}
\label{lem:coefficient-transport}
Let $\varphi:\mathcal O\to\mathcal O'$ be an isomorphism of
finite-dimensional commutative local principal $R$-algebras. Then
$\varphi(\rad^a\mathcal O)=\rad^a\mathcal O'$ for every $a\geq0$.
Consequently, if $0<q_1<\cdots<q_s$, then transport of scalars along
$\varphi$ induces an idempotent-preserving $R$-algebra isomorphism
$$
 \End_{\mathcal O}\left(
   \bigoplus_{i=1}^s\mathcal O/\rad^{q_i}\mathcal O\right)
 \simeq
 \End_{\mathcal O'}\left(
   \bigoplus_{i=1}^s\mathcal O'/\rad^{q_i}\mathcal O'\right).
$$
\end{lemma}

\begin{proof}
Since $\varphi(\rad\mathcal O)=\rad\mathcal O'$, the first assertion
follows by taking powers. Let
$\bar\varphi_a:\mathcal O/\rad^a\mathcal O\to
\mathcal O'/\rad^a\mathcal O'$ be the induced map and set
$\Phi=\bigoplus_i\bar\varphi_{q_i}$. Then $\Phi(ax)=\varphi(a)\Phi(x)$.
Hence conjugation by $\Phi$ induces the asserted $R$-algebra isomorphism and
sends each summand projection to the corresponding summand projection.
\end{proof}

Next we record the local structure.
Let $\cO$ be a finite-dimensional commutative local principal $R$-algebra
of Loewy length $\ell$. Write $\rad\cO=(\pi)$ and
$k=\cO/(\pi)$. For
$P=\{p_1<\cdots<p_s\}\subseteq\{1,\ldots,\ell\}$, define
$M(i)=\cO/(\pi^{p_i})$, $M_P=\bigoplus_{i=1}^sM(i)$, and $\Lambda_{\cO}(P)=\End_{\cO}(M_P)$.
For a positive gap word $g$ with $\sum_i g_i=\ell$, define
$A_{\mathcal O,g}:=\Lambda_{\mathcal O}(P(g))$ and write $A_g$ when
$\mathcal O$ is fixed. By Lemma~\ref{lem:coefficient-transport}, this algebra
depends only on the $R$-algebra isomorphism class of $\mathcal O$ and the
gap word $g$, including the order of its entries.

Since every summand of $M_P$ is annihilated by $\pi^{p_s}$, replacing
$\cO$ by the local principal $R$-algebra $\cO/(\pi^{p_s})$ does not change
$\Lambda_{\cO}(P)$. Thus we may assume $p_s=\ell$ and $M(s)=\cO$.
Multiplication by $\pi$ on every summand defines a central nilpotent
endomorphism $z_\pi$. Since a central nilpotent element belongs to the
Jacobson radical,
$z_\pi\in Z(\Lambda_{\cO}(P))\cap\rad\Lambda_{\cO}(P)$.

Let $Q_s$ be the quiver with vertex set $\{1,\ldots,s\}$ and arrows
$\alpha_i:i\to i+1$ and $\beta_i:i+1\to i$ for $1\leq i<s$. Thus
$Q_s$ is the doubled line
$$
 {
 \underset{1}{\bullet}
 \underset{\beta_1}{\overset{\alpha_1}{\rightleftarrows}}
 \underset{2}{\bullet}\rightleftarrows\cdots\rightleftarrows
 \underset{s-1}{\bullet}
 \underset{\beta_{s-1}}{\overset{\alpha_{s-1}}{\rightleftarrows}}
 \underset{s}{\bullet},}
$$
and set $
 B_s(k):=kQ_s/
 \langle\beta_i\alpha_i,\alpha_i\beta_i\mid1\leq i<s\rangle.$

The following result shows that, after central reduction, the quotient depends
on $P$ only through $s=|P|$ and on $\cO$ only through its residue field.

\begin{proposition}[{see \cite[Theorem~3.1]{HLZtau}}]
\label{prop:uniform-quotient}
There is an idempotent-preserving $k$-algebra isomorphism
$$
 \Lambda_{\cO}(P)/(z_\pi)\simeq B_s(k).
$$
In particular, this quotient depends on $P$ only through $s=|P|$.
\end{proposition}

Let $A:=\Lambda_{\cO}(P)$, let $e_s$ correspond to $M(s)=\cO$, and set
$Q:=Ae_s$. The next lemma applies the standard generator--cogenerator
argument to the longest summand over $\cO$.

\begin{lemma}\label{lem:local-faithful-pi}
The left $A$-module $Q$ is indecomposable, faithful, projective, and
injective.
\end{lemma}

\begin{proof}
The module $Q=Ae_s$ is projective, and
$\End_A(Q)\simeq(e_sAe_s)^{\op}\simeq\cO$.
Since $\cO$ is local, $\End_A(Q)$ is local, and hence $Q$ is indecomposable.

To prove injectivity, write $D=\Hom_R(-,R)$ and choose
$\lambda\in D\cO$ whose restriction to
$\operatorname{soc}\cO=(\pi^{\ell-1})$ is nonzero. Since every nonzero
ideal of $\cO$ contains the socle, for each $0\ne a\in\cO$ there is
$b\in\cO$ such that $\lambda(ab)\ne0$. Hence
the map $\cO\longrightarrow D\cO$, $a\longmapsto(b\longmapsto\lambda(ab))$,
is injective. It is a bimodule map because $\cO$ is commutative, and it is
an isomorphism by finite dimensionality. Therefore, as left $A$-modules,
$$
 \begin{aligned}
 D(e_sA)&\simeq D\Hom_{\cO}(M_P,\cO)
 \simeq D\Hom_{\cO}(M_P,D\cO)
 \simeq D^2M_P
 \simeq M_P
 \simeq\Hom_{\cO}(\cO,M_P)
 \simeq Ae_s.
 \end{aligned}
$$
Since $e_sA$ is projective as a right $A$-module,
$Q\simeq D(e_sA)$ is injective.

To prove faithfulness, identify $Q$ with $\Hom_{\cO}(\cO,M_P)$.
For every $m\in M_P$, the homomorphism $u_m:\cO\to M_P$ defined by
$u_m(r)=rm$ satisfies $u_m(1)=m$. Thus an element of
$A=\End_{\cO}(M_P)$ annihilates $Q$ only if it annihilates $M_P$.
Therefore $\ann_AQ=0$, and $Q$ is faithful.
\end{proof}

\section{Tilting posets of centralizer matrix algebras}
\label{sec:global}
{
As a first step toward the fixed-source realization problem, we determine the
tilting poset of the original algebra. We prove
Theorem~\ref{thm:intro-poset} by computing the tilting poset of one primary
block and extending the formula to arbitrary centralizer matrix algebras via primary
decomposition. An atom graph determines all nontrivial weak-order factors
from the abstract poset, while the center determines the number of remaining
point factors. After the proof, we describe the limitations of this
reconstruction and deduce derived invariance.
}
\begin{proposition}\label{prop:local-tilting}
Let $\cO$ be a finite-dimensional commutative local principal $R$-algebra,
write $\rad\cO=(\pi)$, let $\operatorname{LL}(\cO)$ denote its Loewy length,
and let
$P=\{p_1<\cdots<p_s\}$ be a nonempty set of positive integers with
$p_s\leq\operatorname{LL}(\cO)$. Then
$$
 \tiltA\Lambda_{\cO}(P)\simeq\Weak(\Sigma_s),
 \ \ |\tiltA\Lambda_{\cO}(P)|=s!.
$$
Hence the tilting poset depends only on $s=|P|$; in particular,
it is independent of the Loewy length, the exponent values, and the residue
field.
\end{proposition}

\begin{proof}
Replacing $\cO$ by $\cO/(\pi^{p_s})$ does not change
$\Lambda_{\cO}(P)$, so we may assume $p_s=\operatorname{LL}(\cO)$. Set
$k=\cO/(\pi)$ and $A=\Lambda_{\cO}(P)$, and let $e_s$ correspond to the
summand $\cO/(\pi^{p_s})=\cO$. Assume $s\geq2$ and let $C:=A/Ae_sA$. By
Proposition~\ref{prop:faithful-slice}, $\tiltA A\simeq\stt C$.
Let $\bar z_\pi$ be the image of $z_\pi$ in $C$. Then
$\bar z_\pi\in Z(C)\cap\rad C$.
{
Since
the isomorphism in Proposition~\ref{prop:uniform-quotient} preserves the
vertex idempotents,
$$
 C/(\bar z_\pi)
 \simeq A/(Ae_sA+(z_\pi))
 \simeq B_s(k)/B_s(k)e_sB_s(k)
 \simeq B_{s-1}(k).
$$
}
Combining Proposition~\ref{prop:central-reduction-labels} with
\cite[Corollary~3.2]{HLZtau}, we obtain
$$
 \tiltA A\simeq\stt C
 \simeq\stt(C/(\bar z_\pi))
 \simeq\stt B_{s-1}(k)
 \simeq\Weak(\Sigma_s).
$$
{
The cited result is stated for the left weak order; inversion identifies it
with the right weak order used here.
}

If $s=1$, then $A=\cO$ is local and self-injective, so every tilting module
is projective and $A$ is the unique basic tilting module.
The formula follows for all $s$, and $|\tiltA A|=s!$.
\end{proof}

In the saturated split case, the $\Sigma_s$-parametrization of basic tilting
modules goes back to \cite{BHRR}.

\begin{proof}[Proof of Theorem~\ref{thm:intro-poset}]
Let $f\in R[x]$ be monic and irreducible, and let
$P=\{p_1<\cdots<p_s\}$. Since $f$ is irreducible,
$\cO_f:=R[x]/(f^{p_s})$ is a finite-dimensional commutative local principal
$R$-algebra with radical $(f)/(f^{p_s})$.
Every summand defining $\Lambda_f(P)$ in
\eqref{eq:basic-centralizer} is annihilated by $f^{p_s}$; hence its
$R[x]$-module structure factors through the quotient $\cO_f$. Hence
Proposition~\ref{prop:local-tilting} gives
$\tiltA\Lambda_f(P)\simeq\Weak(\Sigma_{|P|})$.
Since tilting posets are Morita invariant and commute with finite
products, we obtain
\begin{equation}\label{eq:global-tilting-poset}
 \tiltA S_n(c,R)\simeq
 \prod_{f\in\Irr(c)}\Weak(\Sigma_{\kappa_c(f)}).
\end{equation}

Since the weak order of a finite Coxeter group is a lattice
(see \cite[Section~3.2]{BjornerBrenti}),
{
the following graph is defined for every
factor in \eqref{eq:global-tilting-poset}.}
If $L$ is a finite lattice with
minimum $\hat0$, let $G(L)$ be the graph on the atoms of $L$ in which
distinct $a,b$ are adjacent if $|[\hat0,a\vee b]|=6$.

We claim that $G(\Weak(\Sigma_s))$ is a path with $s-1$ vertices for
$s\geq2$ and is empty for $s=1$. Indeed, the atoms are
$s_1,\ldots,s_{s-1}$. If $|i-j|>1$, then $s_i,s_j$
commute and $[1,s_i\vee s_j]=\{1,s_i,s_j,s_is_j\}$. If $|i-j|=1$, the
braid relation gives
$[1,s_i\vee s_j]=\{1,s_i,s_j,s_is_j,s_js_i,s_is_js_i\}$.
Thus the atom graph is the path $s_1-\cdots-s_{s-1}$. Atoms from distinct
product factors have a four-element interval below their join, so the product
graph is the disjoint union of these paths. Consequently, if
$L=\prod_{j=1}^q\Weak(\Sigma_{r_j})$ with $r_j\geq2$, then the connected
component sizes of $G(L)$ are $r_1-1,\ldots,r_q-1$, with multiplicities.
It follows that the connected-component sizes of the atom graph of
\eqref{eq:global-tilting-poset} are precisely
$\ms{\kappa_c(f)-1\mid\kappa_c(f)\geq2}$.
Consequently, the abstract poset determines the multiset of all
$\kappa_c(f)\geq2$. Conversely, \eqref{eq:global-tilting-poset} shows that
this multiset determines the poset since $\Weak(\Sigma_1)$ is a point.

Finally, for $r\geq1$, define
$m_r(c)=|\{f\in\Irr(c)\mid\kappa_c(f)=r\}|$ and $t(c)=|\Irr(c)|$.
The poset determines $m_r(c)$ for every $r\geq2$. By
Proposition~\ref{prop:center-matrix}, the primitive central idempotents of
$Z(S_n(c,R))$ correspond to its local factors, so the center determines
$t(c)$. Therefore
$m_1(c)=t(c)-\sum_{r\geq2}m_r(c)$.
Hence the center together with the poset determines the full multiset
$\ms{\kappa_c(f)}$. This completes the proof.
\end{proof}

For instance, $S_1([0],R)=R$ and
$S_2(\operatorname{diag}(0,1),R)=R\times R$ both have singleton tilting
posets, although their exponent-count multisets are $\ms{1}$ and
$\ms{1,1}$.

The next corollary shows that, although the tilting poset does not
characterize derived equivalence, it is constant on every derived-equivalence
class of centralizer matrix algebras.

\begin{corollary}\label{cor:derived-invariance}
If two centralizer matrix algebras are derived equivalent, then their
tilting posets are isomorphic.
\end{corollary}

\begin{proof}
By Li--Xi's criterion (see \cite[Theorem~1.1 and
Definition~3.1(2)]{LiXiD}), corresponding primary blocks have the same
gap multisets and hence the same numbers of distinct exponents. Therefore
Theorem~\ref{thm:intro-poset} implies the asserted poset isomorphism.
\end{proof}

We end this section with a remark on the scope of
Theorem~\ref{thm:intro-poset} and on the order of the reductions used in its
proof.

\begin{remark}
The order of the reductions in Proposition~\ref{prop:local-tilting} is
essential. Let $\Gamma_2(k)$ be the endomorphism algebra of the
$k[t]/(t^2)$-module $k[t]/(t)\oplus k[t]/(t^2)$, and let $z$ be
multiplication by $t$ on both summands. Then
$\Gamma_2(k)/(z)\simeq B_2(k)$. By
Proposition~\ref{prop:local-tilting}, $\Gamma_2(k)$ has two basic tilting
modules, whereas $B_2(k)$ is
self-injective and its regular module is the unique basic tilting
module. Thus central reduction preserves the support $\tau$-tilting poset
here but not the tilting slice.

Theorem~\ref{thm:intro-poset} provides only partial reconstruction, unlike
the reconstruction of path algebras from tilting posets in
\cite{HappelUngerReconstruction}. The support $\tau$-tilting poset in
\cite[Theorem~1.2]{HLZtau} has a factor
$\Weak(\Sigma_{\kappa_c(f)+1})$ for every primary block and therefore
detects the blocks with $\kappa_c(f)=1$; these contribute only points to
the tilting poset. The center recovers their total number, but the center
and tilting poset do not canonically pair the values $\kappa_c(f)$ with
the local center factors or recover the intermediate exponents. Indeed,
isomorphic local center factors need not have the same number of distinct
exponents. The center--$HH_0$ invariant in \cite[Theorem~1.3]{HLZtau}
recovers the gap multiset, whereas Section~\ref{sec:gaps} uses a
gap word and realizes each of its permutations over a fixed source.
\end{remark}

\section{Fixed-source realization for primary centralizer matrix algebras}
\label{sec:gaps}

In this section, we prove Theorem~\ref{thm:intro-fixed}. Although
Section~\ref{sec:global} determines the tilting poset, it does not recover
the gap word of a primary block. Fix a positive gap word
$g=(g_1,\ldots,g_s)$. For every $w\in\Sigma_s$, we construct a basic
tilting module $T_w$ over the original algebra $A_g$ such that
$\End_{A_g}(T_w)^{\op}\simeq A_{g\cdot w}$.
Subsection~\ref{subsec:adjacent-gaps} realizes a single adjacent gap
exchange and tracks its label. Subsection~\ref{subsec:fixed-source} then
transports these local exchanges to the fixed source and proves the theorem
for every permutation $w$.

Retain the notation of Section~\ref{sec:prelim}. Set
$q_i=\sum_{j=1}^ig_j$ and $r=q_s$. Let $\cO$ be a finite-dimensional
commutative local principal $R$-algebra of Loewy length $r$, and choose
$\pi$ with $\rad\cO=(\pi)$. Set $M(a)=\cO/(\pi^a)$, $M(0)=0$,
$M_g=\bigoplus_{i=1}^sM(q_i)$, and $A_g=\End_{\cO}(M_g)$.
Let $e_j^g$ be the primitive idempotent attached to $M(q_j)$. The symmetric
group acts on gap words on the right by $(g\cdot w)_j=g_{w(j)}$. Thus
$s_i$ exchanges the $i$th and $(i+1)$st gaps, and
$(g\cdot u)\cdot v=g\cdot(uv)$.

\subsection{Adjacent gap exchanges}
\label{subsec:adjacent-gaps}

We first construct an idempotent-preserving isomorphism from each algebra
$A_h$ to its opposite, and then realize adjacent gap exchanges as labeled
irreducible tilting mutations over $A_h$.

\begin{lemma}\label{lem:self-opposite}
For every positive gap word $h$ with total sum $r$, there is an
idempotent-preserving algebra isomorphism
$\iota_h:A_h\xrightarrow{\ \sim\ }A_h^{\op}$.
Thus a right $A_h$-module can be transported to a left $A_h$-module without
changing the labels of the standard projectives.
\end{lemma}

\begin{proof}
Notice that the proof of Lemma~\ref{lem:local-faithful-pi} gives an $\cO$-bimodule
isomorphism $D\cO\simeq\cO$. Hence
$DM(t)\simeq\Hom_{\cO}(M(t),\cO) \simeq(\pi^{r-t})\simeq M(t)$.
Choose a block-diagonal isomorphism $\delta_h:DM_h\xrightarrow{\sim}M_h$
from these summandwise isomorphisms, and define
$\iota_h(\alpha)=\delta_hD(\alpha)\delta_h^{-1}$ for $\alpha\in A_h$.
Since $D(\alpha\beta)=D(\beta)D(\alpha)$, the map $\iota_h$ is an
anti-isomorphism and hence an algebra isomorphism $A_h\to A_h^{\op}$.
Since $\delta_h$ is block diagonal, $\iota_h(e_j^h)=e_j^h$ for every $j$.
\end{proof}

\begin{proposition}\label{lem:adjacent-gap}
Let $1\leq i<s$. Set $b=q_{i-1}$, $d=q_i$, $c=q_{i+1}$, and
$a=b+g_{i+1}$, where $q_0=0$. If
$N=\bigoplus_{j\neq i}M(q_j)$, then there is a minimal
$\add(N)$-split exact sequence
\begin{equation}\label{eq:gap-exchange}
 0\longrightarrow M(d)\longrightarrow
 M(b)\oplus M(c)\longrightarrow M(a)\longrightarrow0.
\end{equation}
This exact sequence produces a basic tilting $A_g$-module $L_i(g)$ with projective
dimension at most one. Moreover,
$\End_{A_g}(L_i(g))^{\op}\simeq A_{g\cdot s_i}$. The module $L_i(g)$ is
the irreducible tilting mutation of the regular module at $A_ge_i^g$. The
endomorphism-ring isomorphism preserves all unchanged labels. It assigns
label $i$ to the new summand.
\end{proposition}

\begin{proof}
\emph{Step 1. Construction and exactness of the exchange sequence.}
We have $b<a<c$, $b<d<c$, and $a+d=b+c$.
Let $p_{u,v}:M(u)\twoheadrightarrow
M(v)$ be the canonical projection for $u\geq v$. For $u\leq v$, let
$j_{u,v}:M(u)\hookrightarrow M(v)$ be given by
$j_{u,v}(\bar z)=\overline{\pi^{v-u}z}$. Define
$F:=(-p_{d,b},j_{d,c})^{\mathsf T}$ and $G:=(j_{b,a},p_{c,a})$.
Since $a-b=c-d$, we have
$GF=-j_{b,a}p_{d,b}+p_{c,a}j_{d,c} =-\pi^{a-b}+\pi^{c-d}=0$.
Since $j_{d,c}$ is injective and $p_{c,a}$ is surjective, $F$ is injective
and $G$ is surjective. Moreover, $M(t)$ has composition length $t$. Hence
$\ker G$ and $\operatorname{Im}F$ both have length
$b+c-a=d$, and $GF=0$ implies that \eqref{eq:gap-exchange} is exact.

\emph{Step 2. Approximation properties and minimality.}
Since a homomorphism $M(u)\to M(v)$ is determined by the image of $1$,
$\Hom_{\cO}(M(u),M(v)) \simeq\{y\in M(v)\mid \pi^uy=0\}$.
Let $M(t)$ be a summand of $N$. Then $t\leq b$ or $t\geq c$. Therefore
$$
\begin{aligned}
 \Hom_{\cO}(M(d),M(t))
 &=
 \begin{cases}
  \Hom_{\cO}(M(b),M(t))\,p_{d,b},&t\leq b,\\
  \Hom_{\cO}(M(c),M(t))\,j_{d,c},&t\geq c,
 \end{cases}\\
 \Hom_{\cO}(M(t),M(a))
 &=
 \begin{cases}
  j_{b,a}\,\Hom_{\cO}(M(t),M(b)),&t\leq b,\\
  p_{c,a}\,\Hom_{\cO}(M(t),M(c)),&t\geq c.
 \end{cases}
\end{aligned}
$$
Hence $F$ is a left $\add(N)$-approximation and $G$ is a right
$\add(N)$-approximation. This is the local-principal analogue of
\cite[Lemma~2.11]{LiXiD}, with $a$ and $d$ interchanged.

For $0\neq\alpha:M(u)\to M(v)$, define $v_\pi(\alpha)$ by
$\alpha(1)\in\pi^{v_\pi(\alpha)}M(v)\setminus
\pi^{v_\pi(\alpha)+1}M(v)$, and set $v_\pi(0)=\infty$. If $b>0$, then
$$
\begin{aligned}
 v_\pi(p_{d,b})&=0<c-d\leq v_\pi(\beta j_{d,c}),
   &&\beta:M(c)\to M(b),\\
 v_\pi(j_{d,c})&=c-d<c-b\leq v_\pi(\gamma p_{d,b}),
   &&\gamma:M(b)\to M(c),\\
 v_\pi(j_{b,a})&=a-b<c-b\leq v_\pi(p_{c,a}\varepsilon),
   &&\varepsilon:M(b)\to M(c),\\
 v_\pi(p_{c,a})&=0<a-b\leq v_\pi(j_{b,a}\theta),
   &&\theta:M(c)\to M(b).
\end{aligned}
$$
Thus no component of $F$ or $G$ factors through the other component. To
prove that $F$ is left minimal, let $U=(u_{vw})$ be an endomorphism of
$M(b)\oplus M(c)$ satisfying $UF=F$.
If $u_{bb}$ were noninvertible, then
$1-u_{bb}$ would be invertible because $\End_{\cO}(M(b))$ is local. The
first component of $UF=F$ would then force $p_{d,b}$ to factor through
$j_{d,c}$, contrary to the first inequality above. Hence $u_{bb}$ is
invertible. The second inequality similarly implies that $u_{cc}$ is
invertible. Since every morphism between the two nonisomorphic
indecomposables lies in the radical, $U$ is invertible. Thus $F$ is left
minimal.
{
The third and fourth inequalities give the dual argument
for the right minimality of $G$.
}

If $b=0$, then the middle term is $M(c)$. For
$u\in\End_{\cO}(M(c))$, the equality $uj_{d,c}=j_{d,c}$ implies
$u\equiv1\pmod{\pi^d}$. Likewise, $p_{c,a}u=p_{c,a}$ implies
$u\equiv1\pmod{\pi^a}$. Since $a,d>0$, both congruences make $u$ a unit in
$\End_{\cO}(M(c))\simeq\cO/(\pi^c)$. Hence both approximations are minimal.
Furthermore, $N\oplus M(d)=M_g$ and
$N\oplus M(a)=M_{g\cdot s_i}$.

\emph{Step 3. The tilting module and its endomorphism algebra.}
Write $V=M_g$, and let $f:M(d)\to M(b)\oplus M(c)$ be the first map in
\eqref{eq:gap-exchange}.
{
Hu--Xi use the opposite convention for composition.}
Thus the endomorphism ring of $V$ in their convention is our $A_g^{\op}$.
In our convention, their Hom construction produces the following right
$A_g$-module:
$$
 \widetilde L_i(g)=\Hom_{\cO}(V,N)\oplus\widetilde C_i,
 \ \
 \widetilde C_i=\operatorname{Coker}\Hom_{\cO}(V,f).
$$
Its new summand has the minimal projective presentation
\begin{equation}\label{min-pro-pr}
 0\longrightarrow\Hom_{\cO}(V,M(d))
 \longrightarrow\Hom_{\cO}(V,M(b)\oplus M(c))
 \longrightarrow\widetilde C_i\longrightarrow0.
\end{equation}
The equivalence
$\Hom_{\cO}(V,-):\add V\to\mathrm{proj}\text{-}A_g$ sends $f$ to a
minimal left approximation by the unchanged projectives. Transport along
$\iota_g$ preserves this property. Let $C_i$ be the image of
$\widetilde C_i$ under this transport. The resulting left module is
$L_i(g):=(\bigoplus_{j\neq i}A_ge_j^g)\oplus C_i$.
After transport along $\iota_g$, the minimal projective presentation
\eqref{min-pro-pr} gives the exchange triangle in
$\Kb(A_g\text{-}\mathrm{proj})$.
By \cite[Definitions~2.30 and~2.34, Proposition~2.33(a)]{AiharaIyama},
$L_i(g)$ is the irreducible left, hence downward, mutation of the regular
module at $A_ge_i^g$.

By \cite[Lemma~3.4 and Theorem~3.5]{HuXi}, the $\add(N)$-split sequence
produces a tilting module and
the corresponding endomorphism-ring isomorphism. After transport along
$\iota_g$, ordinary endomorphism rings satisfy
$$
 \End_{A_g}(L_i(g))^{\op}
 \simeq\End_{\mathrm{mod}\text{-}A_g}(\widetilde L_i(g))^{\op}
 \simeq A_{g\cdot s_i}^{\op}
 \xrightarrow{\ \iota_{g\cdot s_i}^{-1}\ }A_{g\cdot s_i}.
$$
\emph{Step 4. Preservation of labels and basicness.}
To track the labels, we inspect the endomorphism-ring map in the proof of
\cite[Lemma~3.4]{HuXi}.
For $j\neq i$, the idempotent of
$\Hom_{\cO}(V,M(q_j))$ maps to the projection onto
$M(q_j)\subset N\oplus M(a)$, while the idempotent of $\widetilde C_i$ maps
to the projection onto $M(a)$. Transport along the two maps $\iota$ fixes
these idempotents. Since the resulting projection idempotents are primitive
and
$a\notin\{q_j\mid j\neq i\}$, the summands of $L_i(g)$ are pairwise
nonisomorphic. Hence $L_i(g)$ is basic, all unchanged labels are preserved,
and the new summand has label $i$.
\end{proof}
\begin{remark}\label{rem:changing-source}
Iterating Proposition~\ref{lem:adjacent-gap} along an adjacent-transposition
factorization of $w$ gives the changing-source chain of
\cite[Lemma~2.12 and Remark~2.13]{LiXiD}. By
\cite[Theorem~6.4]{Rickard}, this chain realizes $A_{g\cdot w}$ by a tilting
complex over $A_g$, but it does not ensure a tilting module over the fixed source.
\end{remark}

\subsection{Fixed-source realization}
\label{subsec:fixed-source}

We transfer the labeled weak-order model of the saturated Auslander algebra
to $A_g$ through a common quotient. Proposition~\ref{lem:adjacent-gap} then
determines the opposite endomorphism algebras inductively along a spanning
tree of the weak-order graph.

Let $k=\cO/\rad\cO$. Consider the Auslander algebra
$$
 \Gamma_s(k)=\End_{k[t]/(t^s)}
 \left(\bigoplus_{j=1}^sk[t]/(t^j)\right).
$$
Applying Proposition~\ref{prop:uniform-quotient} to $A_g$ and to
$\Gamma_s(k)$ gives idempotent-preserving quotient maps
$A_g\to B_s(k)\leftarrow\Gamma_s(k)$. By
Proposition~\ref{prop:central-reduction-labels}, these maps induce isomorphisms
of the labeled two-term silting posets. Since the last projectives correspond,
Proposition~\ref{prop:faithful-slice} restricts these poset isomorphisms to the
tilting slices.
Hence the reductions preserve every mutation label and
direction.

\begin{proposition}\label{lem:labeled-model}
There is a family
$\{T_w=T_{w,1}\oplus\cdots\oplus T_{w,s}\mid w\in\Sigma_s\}$ forming a
complete set of representatives for $\tiltA A_g$, with $T_1=A_g$ and
$T_{w,s}\simeq A_ge_s^g$. If $\ell(ws_i)=\ell(w)+1$, then $T_{ws_i}$ is a
lower cover of $T_w$, obtained by downward irreducible mutation at the
summand with label $i$. Its exchange triangle is defined by a minimal left
approximation. Thus the parametrization reverses the generation order and
right multiplication by $s_i$ records the mutation label.
\end{proposition}

\begin{proof}
Let $U_v$ denote the vertex indexed by $v$ in the right-module
parametrization of Iyama--Zhang. This parametrization is order reversing,
and mutation with label $i$ is left multiplication by $s_i$
(see \cite[Proposition~3.10, Theorem~3.18, and Corollary~3.19]{IyamaZhang}).
Let $\Theta$ be the composite label-preserving poset isomorphism from this
right-module tilting slice to $\tiltA A_g$. It is induced by
Lemma~\ref{lem:self-opposite} and the two quotient maps. Indeed, transport
along an algebra isomorphism is exact and preserves factor-module
inclusions, while Proposition~\ref{prop:central-reduction-labels} preserves
the generation order. Choose $T_w$ to represent $\Theta(U_{w^{-1}})$, with
$T_1=A_g$. Then
$$
 \ell(ws_i)=\ell(w)+1
 \Leftrightarrow
 \ell(s_iw^{-1})=\ell(w^{-1})+1,
 \ \  (ws_i)^{-1}=s_iw^{-1}.
$$
Thus $T_{ws_i}\lessdot T_w$, and the mutation has label $i$.
{
The support
$\tau$-tilting--two-term silting correspondence identifies this cover with
the corresponding two-term silting mutation
(see \cite[Theorem~2.7(b) and Corollary~2.8]{DIJ}).}
Moreover,
$\Kb(A_g\text{-}\mathrm{proj})$ satisfies condition \textup{(F)} by
\cite[Proposition~2.20]{AiharaIyama}. Since there is no intermediate
silting object, \cite[Theorem~2.35]{AiharaIyama} shows that this is an
irreducible left mutation defined by a minimal left approximation. Finally,
every vertex of the tilting slice contains the last projective, and all
maps defining $\Theta$ preserve its label. Hence
$T_{w,s}\simeq A_ge_s^g$ for every $w$.
\end{proof}

Since each $T_w$ is a tilting $A_g$-module, every module in
$\add\!\left(\bigoplus_{w\in\Sigma_s}T_w\right)$ has finite projective
dimension. We therefore regard these modules, via their minimal projective
resolutions, as objects of $\Kb(A_g\text{-}\mathrm{proj})$; in particular,
this convention applies to all terms of the exchange triangles used below.

The following theorem contains Theorem~\ref{thm:intro-fixed} and also
records all tilting vertices and the primitive-idempotent labels.

\begin{theorem}
\label{thm:fixed-primary}
There are basic tilting modules
$T_w=T_{w,1}\oplus\cdots\oplus T_{w,s}$, one for each $w\in\Sigma_s$,
representing all elements of $\tiltA A_g$, and $T_1=A_g$. If
$\ell(ws_i)=\ell(w)+1$, then $T_{ws_i}$ is the downward irreducible
mutation of $T_w$ at the summand with label $i$. Moreover, there are algebra
isomorphisms
$$
 \phi_w:\End_{A_g}(T_w)^{\op}\xrightarrow{\ \sim\ }A_{g\cdot w}
 \;(w\in\Sigma_s).
$$
For every $j$, the map $\phi_w$ sends the primitive idempotent attached to
$T_{w,j}$ to $e_j^{g\cdot w}$.
\end{theorem}

\begin{proof}
For each $v\neq1$, choose a right descent $i(v)$ and set
$p(v):=vs_{i(v)}$, so $\ell(p(v))=\ell(v)-1$. Hence the edges
$p(v)\to v$ form a spanning tree rooted at $1$.
Proposition~\ref{lem:labeled-model} fixes the modules $T_w$. We use this
tree to construct the maps
$\phi_w$ by induction on $\ell(w)$.

At the root, let $\phi_1:\End_{A_g}(A_g)^{\op}\xrightarrow{\sim}A_g$ be
the standard idempotent-preserving isomorphism. Assume that $\phi_w$
has been constructed. Write
$B_w:=\End_{A_g}(T_w)^{\op}$ and $h:=g\cdot w$, and define $F_w:=(\phi_w)_*\circ\mathbf R\!\Hom_{A_g}(T_w,-)$.
Here $(\phi_w)_*$ denotes transport of scalars along the algebra isomorphism
$\phi_w$.
Our right-action convention implies
$h\cdot s_i=(g\cdot w)\cdot s_i=g\cdot(ws_i)$.
By Happel's tilting theorem (see \cite{HappelDerived}), $F_w$ is a triangle equivalence, and the
label condition on $\phi_w$ implies
$F_w(T_{w,j})\simeq A_he_j^h$ for every $j$.

Consider a tree edge $w\to ws_i$. By
Proposition~\ref{lem:labeled-model}, the exchange triangle
$$
 T_{w,i}\longrightarrow U'\longrightarrow Y_i
 \longrightarrow T_{w,i}[1],
  \ \ U'\in\add\!\left(\bigoplus_{j\neq i}T_{w,j}\right),
$$
is defined by a minimal left approximation and
$T_{ws_i}\simeq(\bigoplus_{j\neq i}T_{w,j})\oplus Y_i$. Since $F_w$ is an
equivalence, it sends this triangle to an exchange triangle
$$
 A_he_i^h\longrightarrow F_w(U')\longrightarrow F_w(Y_i)
 \longrightarrow A_he_i^h[1]
$$
whose first map is a minimal left approximation by the unchanged
projectives. Such approximations, and hence their cones, are unique up to
isomorphism. Thus Proposition~\ref{lem:adjacent-gap} gives
$F_w(T_{ws_i})\simeq L_i(h)$ in $\Kb(A_h\text{-}\mathrm{proj})$.
Choose this isomorphism summandwise, matching the unchanged summands
according to their labels and identifying $F_w(Y_i)$ with the new summand
of $L_i(h)$.
By full faithfulness and Proposition~\ref{lem:adjacent-gap}, it follows that
$$
 \End_{A_g}(T_{ws_i})^{\op}
 \simeq\End_{A_h}(L_i(h))^{\op}
 \simeq A_{h\cdot s_i}=A_{g\cdot ws_i}.
$$
Both isomorphisms preserve every unchanged label and the new label $i$.
Hence their composite is the required $\phi_{ws_i}$, and induction proves
the theorem.
\end{proof}

\section{Global fixed-source realization for centralizer matrix algebras}
\label{sec:one-step}

{
We now extend Theorem~\ref{thm:fixed-primary} from primary basic blocks to
arbitrary centralizer matrix algebras. Morita lifting adjusts the multiplicities of
the tilting summands to realize prescribed nonbasic algebras up to isomorphism.
Together with Li--Xi's classification, this proves
Theorem~\ref{thm:intro-one-step}. We then derive its HRS-tilting,
extension-dimension, and exhaustion consequences.
}

\begin{lemma}\label{lem:Morita-lifting}
Let $X$ be a tilting $A$-module, and write
$C=\End_A(X)^{\op}$. If an algebra $D$ is Morita equivalent to
$C$, then there is $T\in\add X$ such that
$\add T=\add X$ and $\End_A(T)^{\op}\simeq D$. In particular, $T$ is a
tilting $A$-module.
\end{lemma}

\begin{proof}
Choose a finitely generated projective generator $P\in C\text{-}\mathrm{mod}$
with $\End_C(P)^{\op}\simeq D$. The equivalence
$F:=\Hom_A(X,-):\add X\xrightarrow{\sim}C\text{-}\mathrm{proj}$
allows us to choose $T\in\add X$ with $F(T)\simeq P$. Since $P$ is a projective generator,
$\add P=C\text{-}\mathrm{proj}=\add C=\add F(X)$,
and hence $\add T=\add X$. By full faithfulness,
$\End_A(T)^{\op}\simeq\End_C(F(T))^{\op} \simeq\End_C(P)^{\op}\simeq D$.
Since $X$ is tilting and $\add T=\add X$, the module $T$ is also tilting.
This proves the assertion.
\end{proof}

\begin{samepage}
\begin{theorem}\label{thm:one-step}
Let $A=S_n(c,R)$ and $B=S_m(d,R)$. The following are equivalent.
\begin{enumerate}[label={\textup{(\arabic*)}}]
\item $A$ and $B$ are derived equivalent.
\item $c$ and $d$ are $D$-equivalent.
\item There is a tilting left $A$-module $T$ such that
$\End_A(T)^{\op}\simeq B$.
\end{enumerate}
If $B$ is basic, then $T$ can be chosen basic. The same assertion holds with
$A$ and $B$ interchanged.
\end{theorem}
\end{samepage}

\begin{proof}
Theorem~\ref{thm:LiXi-criterion} gives
\textup{(1)}$\Leftrightarrow$\textup{(2)}.
If \textup{(3)} holds, Happel's tilting theorem
(see \cite{HappelDerived}) gives a triangle equivalence
$\Db(A)\simeq\Db(\End_A(T)^{\op})\simeq\Db(B)$, and hence
\textup{(1)} holds.

Assume \textup{(1)}. Take $A^{\mathrm b}$ and $B^{\mathrm b}$ to be the
basic centralizer matrix algebras in \eqref{eq:basic-centralizer}, with $B^{\mathrm b}=B$
when $B$ is basic. By Theorem~\ref{thm:LiXi-criterion}, after reindexing the
primary blocks and using Lemma~\ref{lem:coefficient-transport}, we may write
$$
 A^{\mathrm b}\simeq\prod_{\lambda=1}^t
 A_{\mathcal O_\lambda,g^\lambda},
 \ \
 B^{\mathrm b}\simeq\prod_{\lambda=1}^t
 A_{\mathcal O_\lambda,g^\lambda\cdot w_\lambda}
$$
for suitable $w_\lambda$. Theorem~\ref{thm:fixed-primary} and the finite
product property produce a basic $X^{\mathrm b}\in\tiltA A^{\mathrm b}$ with
\(
 \End_{A^{\mathrm b}}(X^{\mathrm b})^{\op}\simeq B^{\mathrm b}.
\)
Transporting $X^{\mathrm b}$ along an $R$-linear Morita equivalence produces a
basic $X\in\tiltA A$ with $\End_A(X)^{\op}\simeq B^{\mathrm b}$. If $B$ is
basic, set $T=X$. Otherwise Lemma~\ref{lem:Morita-lifting} gives
$T\in\add X$ with $\End_A(T)^{\op}\simeq B$. Thus \textup{(3)} holds.

Interchanging $A$ and $B$ proves the final assertion.
\end{proof}

The following example illustrates the fixed-source realization for $\Sigma_3$,
the smallest symmetric group with a nontrivial braid relation, and shows how
centralizer matrix algebras of different sizes arise from one fixed source.

\begin{example}\label{ex:s3-fixed-source}
Let $J_d(0)$ denote the nilpotent Jordan block of size $d$. For the gap
word $(1,2,3)$, the fixed source is
$$
 A_{(1,2,3)}
 =\End_{R[t]/(t^6)}
 \left(R[t]/(t)\oplus R[t]/(t^3)\oplus R[t]/(t^6)\right)
 \simeq
 S_{10}\!\left(J_1(0)\oplus J_3(0)\oplus J_6(0),R\right).
$$
For $w\in\Sigma_3$, let $q_i(w)$ be the $i$th successive partial sum of
$(1,2,3)\cdot w$. Theorem~\ref{thm:fixed-primary} gives
$$
\begin{aligned}
 \End_{A_{(1,2,3)}}(T_w)^{\op}
 &\simeq A_{(1,2,3)\cdot w}\\
 &\simeq
 S_{q_1(w)+q_2(w)+q_3(w)}\!\left(
 J_{q_1(w)}(0)\oplus J_{q_2(w)}(0)\oplus J_{q_3(w)}(0),R\right).
\end{aligned}
$$
After transporting scalars along this isomorphism, the associated
fixed-source triangle equivalence is
$$
 \Db\!\left(A_{(1,2,3)}\right)
 \xrightarrow{\ \mathbf R\!\Hom_{A_{(1,2,3)}}(T_w,-)\ }
 \Db\!\left(\End_{A_{(1,2,3)}}(T_w)^{\op}\right)
 \xrightarrow{\ \sim\ }
 \Db\!\left(A_{(1,2,3)\cdot w}\right).
$$
The table below lists the six parameters, the permuted gap words, the
partial sums, and the dimensions of the corresponding opposite endomorphism
algebras.
\begin{center}
\small
\setlength{\tabcolsep}{7pt}
\renewcommand{\arraystretch}{1.12}
\begin{tabular}{c|c|c|c}
$w$ & $(1,2,3)\cdot w$ & successive partial sums
& $\dim_R A_{(1,2,3)\cdot w}$\\ \hline
$1$ & $(1,2,3)$ & $(1,3,6)$ & $20$\\
$s_1$ & $(2,1,3)$ & $(2,3,6)$ & $25$\\
$s_2$ & $(1,3,2)$ & $(1,4,6)$ & $23$\\
$s_1s_2$ & $(2,3,1)$ & $(2,5,6)$ & $31$\\
$s_2s_1$ & $(3,1,2)$ & $(3,4,6)$ & $33$\\
$s_1s_2s_1$ & $(3,2,1)$ & $(3,5,6)$ & $36$
\end{tabular}
\end{center}
Indeed,
$$
 \dim_R A_{(1,2,3)\cdot w}
 =\sum_{i,j=1}^3\min\{q_i(w),q_j(w)\},
$$
so the last column proves that the six opposite endomorphism algebras are
pairwise nonisomorphic. Figure~\ref{fig:s3-fixed-source-correspondence} shows the
complete tilting poset.
\end{example}

\begin{figure}[!htbp]
\centering
\begin{tikzpicture}[
  vertex/.style={draw=black!65,fill=tiltingfill,rounded corners,
    minimum width=2.9cm,minimum height=7mm,align=center,font=\small},
  edge/.style={->,>=stealth,semithick},
  elabel/.style={fill=white,inner sep=1pt,font=\scriptsize}
]
\node[vertex] (T1) at (0,4.8) {$T_1=A_{(1,2,3)}$};
\node[vertex] (Ts1) at (-1.85,3.2) {$T_{s_1}$};
\node[vertex] (Ts2) at (1.85,3.2) {$T_{s_2}$};
\node[vertex] (Ts1s2) at (-1.85,1.6) {$T_{s_1s_2}$};
\node[vertex] (Ts2s1) at (1.85,1.6) {$T_{s_2s_1}$};
\node[vertex] (Ts1s2s1) at (0,0) {$T_{s_1s_2s_1}$};
\draw[edge,draw=mutationone] (T1) --
  node[elabel,text=mutationone,above left] {$1$} (Ts1);
\draw[edge,draw=mutationtwo] (T1) --
  node[elabel,text=mutationtwo,above right] {$2$} (Ts2);
\draw[edge,draw=mutationtwo] (Ts1) --
  node[elabel,text=mutationtwo,left] {$2$} (Ts1s2);
\draw[edge,draw=mutationone] (Ts2) --
  node[elabel,text=mutationone,right] {$1$} (Ts2s1);
\draw[edge,draw=mutationone] (Ts1s2) --
  node[elabel,text=mutationone,below left] {$1$} (Ts1s2s1);
\draw[edge,draw=mutationtwo] (Ts2s1) --
  node[elabel,text=mutationtwo,below right] {$2$} (Ts1s2s1);
\end{tikzpicture}
\caption{The complete tilting Hasse diagram for the gap word $(1,2,3)$ in
Example~\ref{ex:s3-fixed-source}. Arrows represent downward irreducible
mutations. An arrow labeled $i$ exchanges the entries in positions $i$ and
$i+1$ of the current gap word. The corresponding centralizer matrix
algebras are specified by the table and the uniform formula in the example.}
\label{fig:s3-fixed-source-correspondence}
\end{figure}

We now apply standard tilting theory to derive consequences of the
fixed-source realization. For a tilting left $A$-module $T$, define
$$
 \mathcal T_T:=\Fac T=\{M\in A\text{-}\mathrm{mod}\mid\Ext_A^1(T,M)=0\},
 \;
 \mathcal F_T:=\{N\in A\text{-}\mathrm{mod}\mid\Hom_A(T,N)=0\}.
$$
By the Brenner--Butler theorem (see \cite{BrennerButler}),
$(\mathcal T_T,\mathcal F_T)$ is a torsion pair; its HRS tilt is the heart
$\mathcal H(\mathcal T_T,\mathcal F_T)$ of
Definition~\ref{def:torsion-pair-hrs}.

\begin{corollary}
\label{cor:one-step-hrs}
Let $A$ and $B$ be derived-equivalent centralizer matrix algebras. Then there is a
tilting left $A$-module $T$ such that $B\simeq\End_A(T)^{\op}$ and
$\Phi:=\mathbf R\!\Hom_A(T,-):\Db(A)\xrightarrow{\sim}\Db(B)$
restricts to an exact equivalence
$$
 \mathcal H(\mathcal T_T,\mathcal F_T)
 \xrightarrow{\sim}B\text{-}\mathrm{mod}.
$$
\end{corollary}

\begin{proof}
Choose $T$ as in Theorem~\ref{thm:one-step}. Via
$B\simeq\End_A(T)^{\op}$, regard $T$ as an $A$--$B$-bimodule.
Happel's tilting theorem (see \cite{HappelDerived}) gives the equivalence
$\Phi$ with quasi-inverse $T\otimes_B^{\mathbf L}-$.
Set
$\mathcal X_T:=\ker(T\otimes_B-)$ and $\mathcal Y_T:=\ker\Tor_1^B(T,-)$.
The Brenner--Butler theorem (see \cite{BrennerButler}) gives the torsion pair
$(\mathcal X_T,\mathcal Y_T)$ in $B\text{-}\mathrm{mod}$ and equivalences
$\Hom_A(T,-):\mathcal T_T\xrightarrow{\sim}\mathcal Y_T$ and
$\Ext_A^1(T,-):\mathcal F_T\xrightarrow{\sim}\mathcal X_T$.
Since $\pd_AT\leq1$, one has
$\Phi(M)\simeq\Hom_A(T,M)$ for $M\in\mathcal T_T$, and $\Phi(N[1])\simeq\Ext_A^1(T,N)$ for $N\in\mathcal F_T$.
Hence $\Phi$ sends
$\langle\mathcal F_T[1],\mathcal T_T\rangle_{\mathrm{ex}}$ into
$B\text{-}\mathrm{mod}$. Conversely, every $B$-module is an extension of
an object of $\mathcal Y_T$ by an object of $\mathcal X_T$.
These equivalences place its image under a quasi-inverse of $\Phi$
in $\langle\mathcal F_T[1],\mathcal T_T\rangle_{\mathrm{ex}}$.
Thus $\Phi$ restricts to the asserted exact equivalence of hearts.
\end{proof}

\begin{corollary}
\label{cor:extension-dimension-bound}
Let $A$ and $B$ be derived-equivalent centralizer matrix algebras. Then one has
$$
 \bigl|\extdim(A)-\extdim(B)\bigr|\leq1.
$$
\end{corollary}

\begin{proof}
Choose $T$ as in Theorem~\ref{thm:one-step}, and let $P^\bullet(T)$ be its
deleted minimal projective resolution. It represents the image of $B$
under the quasi-inverse of $\mathbf R\!\Hom_A(T,-)$; see
\cite[Theorem~6.4]{Rickard}.
Since $P^\bullet(T)$ is minimal, its length in the sense of
\cite[p.~18 and Lemma~2.1]{ZhangZhengExt} is
$$
 \operatorname{len}(P^\bullet(T))
 :=\sup\{i\mid P^i(T)\neq0\}-\inf\{i\mid P^i(T)\neq0\}+1\leq2,
$$
because $\pd_AT\leq1$. Applying \cite[Theorem~3.4]{ZhangZhengExt} gives
$$
 \bigl|\extdim(A)-\extdim(B)\bigr|
 \leq\operatorname{len}(P^\bullet(T))-1\leq1.
$$
\end{proof}

\begin{remark}\label{ex:LiXi-extension-dimension}
The bound in Corollary~\ref{cor:extension-dimension-bound} is sharp.
Assume that $R$ is algebraically closed, and set
$c=J_5(0)\oplus J_4(0)\oplus J_1(0)\in M_{10}(R)$ and $d=J_5(0)\oplus J_2(0)\oplus J_1(0)\in M_8(R)$,
where $J_r(0)$ is the nilpotent Jordan block of size $r$.
By \cite[Example~5.3]{LiXiD}, $S_{10}(c,R)$ and $S_8(d,R)$ are derived
equivalent, with the first representation-finite and the second
representation-infinite. Definition~\ref{def:extension-dimension} therefore
gives $\extdim(S_{10}(c,R))=0$ and $\extdim(S_8(d,R))\neq0$, while
Corollary~\ref{cor:extension-dimension-bound} gives
$\extdim(S_8(d,R))\leq1$. Hence $\extdim(S_8(d,R))=1$.
\end{remark}

The following result shows that the class of centralizer matrix algebras is closed under
$R$-linear Morita equivalence.

\begin{lemma}\label{lem:centralizer-Morita-closed}
Let $C$ be a centralizer matrix algebra over $R$. If an $R$-algebra $D$ is
$R$-linearly Morita equivalent to $C$, then $D$ is a centralizer matrix algebra over
$R$.
\end{lemma}

\begin{proof}
Write $C:=\End_{R[x]}(M)$ for a finite-dimensional $R[x]$-module $M$, and
work with right $C$-modules. Since $D$ is $R$-linearly Morita equivalent to
$C$, there is a finitely generated projective generator $P_C$ such that
$
 D\simeq\End_C(P)
$
as $R$-algebras. Under the standard equivalence
$
 \Hom_{R[x]}(M,-):
 \add M\xrightarrow{\sim}\mathrm{proj}\text{-}C,
$
we have $P\simeq\Hom_{R[x]}(M,N)$ for some $N\in\add M$. Full faithfulness
then gives
$
 D\simeq\End_C(P)
 \simeq\End_C\bigl(\Hom_{R[x]}(M,N)\bigr)
 \simeq\End_{R[x]}(N).
$

Since $N\in\add M$, it is finite-dimensional over $R$. Choose an $R$-basis
of $N$ and let $c_N\in M_{\dim_RN}(R)$ represent multiplication by $x$.
Then
$D\simeq\End_{R[x]}(N)\simeq S_{\dim_RN}(c_N,R)$,
as an $R$-linear endomorphism is $R[x]$-linear exactly when it commutes with
the action of $x$. Thus $D$ is a centralizer matrix algebra over $R$.

\end{proof}

We now show that the opposite endomorphism algebras of all tilting modules
over a fixed centralizer matrix algebra are precisely the centralizer matrix algebras in its
derived equivalence class.

\begin{corollary}
\label{cor:closure-exhaustion}
Fix a centralizer matrix algebra $A$. Then
$$
 \bigl\{[\End_A(T)^{\op}]\ \bigm|\
 T\text{ is a tilting }A\text{-module}\bigr\}
 =
 \bigl\{[B]\ \bigm|\
 \substack{B\text{ is a centralizer matrix algebra}\\
 B\text{ is derived equivalent to }A}\bigr\}.
$$
Here brackets denote isomorphism classes of $R$-algebras.
\end{corollary}

\begin{proof}
Write $A=S_n(c,R)$. Let $T$ be tilting, and choose a basic summand
$T^{\mathrm b}$ with
$\add T^{\mathrm b}=\add T$. Let $Y$ be its image under an $R$-linear
Morita equivalence to the basic centralizer matrix algebra
$A^{\mathrm b}\simeq\prod_{\lambda=1}^t A_{\mathcal O_\lambda,g^\lambda}$.
Choose this decomposition so that
$\mathcal O_\lambda=R[x]/(f_\lambda^{r_\lambda})$ for distinct
$f_1,\ldots,f_t\in\Irr(c)$, and write
$g^\lambda=(g_1^\lambda,\ldots,g_{s_\lambda}^\lambda)$.
By the finite-product decomposition and Proposition~\ref{lem:labeled-model}, the
component of $Y$ indexed by $\lambda$ is isomorphic to some $T_{w_\lambda}$.
Hence Theorem~\ref{thm:fixed-primary} gives
$$
 \End_A(T^{\mathrm b})^{\op}
 \simeq\End_{A^{\mathrm b}}(Y)^{\op}
 \simeq\prod_{\lambda=1}^t
 A_{\mathcal O_\lambda,g^\lambda\cdot w_\lambda}.
$$
Set $q_{\lambda,i}:=\sum_{j=1}^i(g^\lambda\cdot w_\lambda)_j$
for $1\leq i\leq s_\lambda$, and let $M_\lambda:=\bigoplus_{i=1}^{s_\lambda}R[x]/(f_\lambda^{q_{\lambda,i}})$.
Then
$A_{\mathcal O_\lambda,g^\lambda\cdot w_\lambda}
\simeq\End_{R[x]}(M_\lambda)$. Since powers of distinct $f_\lambda$ are
coprime,
{
B\'ezout gives
$$
 \Hom_{R[x]}(M_\lambda,M_\mu)=0\quad(\lambda\neq\mu),
  \ \
 \prod_\lambda\End_{R[x]}(M_\lambda)
 \simeq\End_{R[x]}\!\left(\bigoplus_\lambda M_\lambda\right).
$$}

Thus $\End_A(T^{\mathrm b})^{\op}$ is a centralizer matrix algebra. Since
$\add T=\add T^{\mathrm b}$, it follows from
\cite[Lemma~2.2]{LiXiD} that $\End_A(T)^{\op}$ and
$\End_A(T^{\mathrm b})^{\op}$ are $R$-linearly Morita equivalent.
Lemma~\ref{lem:centralizer-Morita-closed} therefore shows
that $\End_A(T)^{\op}$ is a centralizer matrix algebra, while
Happel's tilting theorem (see \cite{HappelDerived}) implies
$\End_A(T)^{\op}\overset{\mathrm D}{\sim}A$. Hence the set on the left
is contained in the set on the right.

Conversely, if $B$ is a centralizer matrix algebra with
$B\overset{\mathrm D}{\sim}A$, Theorem~\ref{thm:one-step} gives a tilting
left $A$-module $T$ with $\End_A(T)^{\op}\simeq B$.
\end{proof}

\section{Fixed-source uniqueness and Schreier multigraphs}
\label{sec:endomorphism-classes}

Theorem~\ref{thm:one-step} settles the global existence problem. We now
turn to uniqueness for a fixed primary source. Fix the algebra $A_g$ and
the family $\{T_w\}_{w\in\Sigma_s}$ constructed in
Theorem~\ref{thm:fixed-primary}.
For a gap word $h=(h_1,\ldots,h_s)$, write
$q_i(h):=\sum_{j=1}^i h_j$.

For each distinct value $a$ occurring among the entries of $g$, set
$I_a:=\{j\in\{1,\ldots,s\}\mid g_j=a\}$ and $m_a:=|I_a|$,
and let
$$
 H_g:=\{h\in\Sigma_s\mid g\cdot h=g\}
 =\prod_a\operatorname{Sym}(I_a)
 \simeq\prod_a\Sigma_{m_a}.
$$
Thus $H_g$ is the stabilizer of the gap word $g$ and independently
permutes the positions carrying each fixed gap value.

Let $\mathcal E_g$ be the set of isomorphism classes of opposite
endomorphism algebras of basic tilting $A_g$-modules, and define
$\Psi_g:\Sigma_s\longrightarrow\mathcal E_g$ by $\Psi_g(w):=[\End_{A_g}(T_w)^{\op}]$.

\begin{samepage}
\begin{theorem}\label{thm:endomorphism-classes}
For $u,v\in\Sigma_s$, the following statements are equivalent.
\begin{enumerate}[label={\textup{(\arabic*)}}]
\item $\End_{A_g}(T_u)^{\op}\simeq\End_{A_g}(T_v)^{\op}$ as
$R$-algebras.
\item $\End_{A_g}(T_u)^{\op}$ and $\End_{A_g}(T_v)^{\op}$ are
$R$-linearly Morita equivalent.
\item $g\cdot u=g\cdot v$.
\item $uv^{-1}\in H_g$.
\item $H_gu=H_gv$.
\end{enumerate}
Each isomorphism class in the image of $\Psi_g$ is represented by
$|H_g|=\prod_a m_a!$ parameters. Hence there are exactly
$s!/\prod_a m_a!$ isomorphism classes.
\end{theorem}
\end{samepage}

\begin{proof}
Theorem~\ref{thm:fixed-primary} gives
\textup{(3)}~$\Rightarrow$~\textup{(1)}, and
\textup{(1)}~$\Rightarrow$~\textup{(2)} is immediate.

\noindent\textup{(2)}~$\Rightarrow$~\textup{(3)}.\quad
The positive gaps make the exponent sequences
$q_1(g\cdot u)<\cdots<q_s(g\cdot u)$ and
$q_1(g\cdot v)<\cdots<q_s(g\cdot v)$ strictly increasing. Thus
$M_{g\cdot u}$ and $M_{g\cdot v}$ are basic modules over the same local
Nakayama algebra $\cO$. If \textup{(2)} holds, then the basicness of these
modules and \cite[Lemma~2.2]{LiXiD} give
$M_{g\cdot u}\simeq M_{g\cdot v}$. By the Krull--Schmidt theorem and the
fact that $\cO/\rad^r\cO\not\simeq\cO/\rad^{r'}\cO$ for $r\neq r'$, their
exponent sets coincide. Since both sequences are strictly increasing,
$q_i(g\cdot u)=q_i(g\cdot v)$ for $1\leq i\leq s$,
and hence $g\cdot u=g\cdot v$.

Finally, one has
$g\cdot u=g\cdot v \Leftrightarrow uv^{-1}\in H_g \Leftrightarrow H_gu=H_gv$.
Thus two parameters have the same image precisely when they lie in the same
left coset of $H_g$, and each such coset has size
$|H_g|=\prod_a m_a!$.
\end{proof}

The number of parameters representing each isomorphism class is independent
of the choice of $h$ in $\Omega_g:=g\cdot\Sigma_s$. Indeed, if $h=g\cdot u$,
then $H_h=u^{-1}H_gu$. Consequently, for any $h,k\in\Omega_g$, exactly $|H_g|$
nonisomorphic basic tilting $A_h$-modules have opposite endomorphism algebras
isomorphic to $A_k$.

We retain the mutation labels and let $\mathcal G_g$ be the undirected
labeled graph with vertices $T_w$ and an edge
$T_w\xleftrightarrow{\ i\ }T_{ws_i}$. By Proposition~\ref{lem:labeled-model},
this is the labeled Hasse graph of $\tiltA A_g$ with the orientation
forgotten. The group $H_g$ acts by $h\cdot T_w=T_{hw}$.

The coset-graph construction goes back to Schreier (see \cite{Schreier}). For
graph-theoretic treatments retaining loops and multiple edges, see
\cite[Section~2]{GrossSchreier} and \cite[p.~403]{Cannizzo}.

For a subgroup $H\leq\Sigma_s$, write $X_H:=H\backslash\Sigma_s$. For
$1\leq i<s$, let
$\rho_i^H:X_H\longrightarrow X_H$ be given by $\rho_i^H(Hw)=Hws_i$.
We denote by
$
 \operatorname{Sch}(H\backslash\Sigma_s;s_1,\ldots,s_{s-1})
$
the undirected labeled multigraph with vertex set $X_H$ and edge set
$
 \bigsqcup_{i=1}^{s-1}X_H/\langle\rho_i^H\rangle.
$
An orbit in the $i$th summand is an edge with label $i$, and a singleton
orbit is a loop. The disjoint union keeps edges with different labels
distinct, so loops and parallel edges are retained. The quotient mutation
graph now has a direct group-theoretic description.

\begin{proposition}
\label{prop:Schreier}
Form the labeled quotient multigraph $H_g\backslash\mathcal G_g$ whose
vertices and edges are the $H_g$-orbits of the vertices and labeled edges of
$\mathcal G_g$, retaining all loops and parallel edges arising in the
quotient. Then
$
 H_g\backslash\mathcal G_g\simeq
 \operatorname{Sch}\bigl(H_g\backslash\Sigma_s;
 s_1,\ldots,s_{s-1}\bigr)
$
as undirected labeled multigraphs. Under this isomorphism, the vertex orbit
of $T_w$ corresponds to $H_gw$, and an $i$-labeled edge exchanges the entries
in positions $i$ and $i+1$ of $g\cdot w$. It is a loop exactly when these
entries are equal.
\end{proposition}

\begin{proof}
Since $(hw)s_i=h(ws_i)$, the $H_g$-action preserves labels and edges.
Moreover,
$g\cdot(hw)=(g\cdot h)\cdot w=g\cdot w$, so it preserves $\Psi_g$.
By Theorem~\ref{thm:endomorphism-classes}, its vertex orbits are precisely
the equivalence classes defined by isomorphic opposite endomorphism algebras.
For a fixed $i$, write
$e_i(w):=\{T_w,T_{ws_i}\}$.
The assignment
$
 H_g\!\cdot e_i(w)\longmapsto\{H_gw,H_gws_i\}
$
is well defined: replacing $w$ by $hw$ with $h\in H_g$, or by $ws_i$,
does not change either side.

Conversely, equality of the corresponding
$\rho_i^{H_g}$-orbits implies that the two edges lie in the same
$H_g$-orbit. Hence the assignment is a bijection from the $H_g$-orbits of
$i$-labeled edges of $\mathcal G_g$ to the orbits of $\rho_i^{H_g}$ on
$X_{H_g}$. Taking the
disjoint union over $i$ gives the asserted isomorphism of labeled
multigraphs.

Finally, one has
$$
\begin{aligned}
 H_gw=H_gws_i
 &\Leftrightarrow ws_iw^{-1}\in H_g\\
 &\Leftrightarrow g\cdot ws_i=g\cdot w\\
 &\Leftrightarrow (g\cdot w)_i=(g\cdot w)_{i+1}.
\end{aligned}
$$
Thus the projected edge is a loop exactly when the exchanged gaps are equal.
\end{proof}

Proposition~\ref{prop:Schreier} forgets the weak-order orientation. By
Definition~\ref{def:right-weak-order},
$$
 w\longrightarrow ws_i
 \quad\Leftrightarrow\quad
 \ell(ws_i)=\ell(w)+1
 \quad\Leftrightarrow\quad w(i)<w(i+1).
$$
Under the parametrization $w\mapsto T_w$, this direction is opposite to the
generation-order orientation of the tilting Hasse diagram. By
Proposition~\ref{prop:Schreier}, the lifts of the quotient $i$-edge
represented by $\{H_gw,H_gws_i\}$ are precisely the edges
$\{T_{hw},T_{hws_i}\}$ with $h\in H_g$. We say that the weak-order
orientation descends if every nonloop quotient edge has the same direction
for all its lifts. A subgroup of $\Sigma_s$ is
\emph{standard parabolic} if it is generated by a subset of
$\{s_1,\ldots,s_{s-1}\}$; see
\cite[Section~2.4]{BjornerBrenti}.

The next theorem characterizes when the weak-order orientation descends.

\begin{theorem}
\label{thm:orientation-descent}
The following statements are equivalent.
\begin{enumerate}[label={\textup{(\arabic*)}}]
\item The right weak-order orientation descends to every nonloop edge of
$H_g\backslash\mathcal G_g$.
\item For any two distinct gap values $a$ and $b$, either every element of
$I_a$ is smaller than every element of $I_b$, or every element of $I_b$ is
smaller than every element of $I_a$.
\item Every set $I_a$ is an interval in $\{1,\ldots,s\}$.
\item The group $H_g$ is a standard parabolic subgroup of $\Sigma_s$ with
respect to $s_1,\ldots,s_{s-1}$.
\end{enumerate}
\end{theorem}
\begin{proof}
\noindent\textup{(2)}~$\Rightarrow$~\textup{(1)}.\quad
Let the quotient edge represented by
$\{H_gw,H_gws_i\}$ be nonloop. Its lifts are
$\{T_{hw},T_{hws_i}\}$ with $h\in H_g$. The nonloop condition means that
$w(i)\in I_a$ and $w(i+1)\in I_b$ for distinct gap values $a,b$. Since
$h(I_a)=I_a$ and $h(I_b)=I_b$, condition \textup{(2)} gives
$w(i)<w(i+1) \Leftrightarrow h(w(i))<h(w(i+1))$.
Thus all lifts have the same direction, proving \textup{(2)}$\Rightarrow$\textup{(1)}.

\noindent\textup{(1)}~$\Rightarrow$~\textup{(3)}.\quad
Suppose that \textup{(3)} fails. Then
$p<q<p'$ for some $p,p'\in I_a$ and $q\in I_b$ with $a\neq b$.
Choose $u$ with $u(1)=p$ and $u(2)=q$, let $h=(p\ p')\in H_g$, and set
$v=hu$. Then $H_gu=H_gv$, while
$u(1)=p<q=u(2)$ and $v(1)=p'>q=v(2)$.
Since $a\neq b$, the edge
$H_gu\xleftrightarrow{\ 1\ }H_gus_1$ is nonloop. Moreover, $v=hu$ implies
that $u\xleftrightarrow{\ 1\ }us_1$ and
$v\xleftrightarrow{\ 1\ }vs_1$ are lifts of this same quotient edge. These
inequalities show that the lifts have opposite weak-order orientations.
Therefore the orientation does not descend, proving
\textup{(1)}$\Rightarrow$\textup{(3)}.

\noindent\textup{(2)}~$\Leftrightarrow$~\textup{(3)}.\quad
For a partition of a linearly ordered set, its parts are pairwise separated
if and only if each part is an interval; hence
\textup{(2)}$\Leftrightarrow$\textup{(3)}.

\noindent\textup{(3)}~$\Leftrightarrow$~\textup{(4)}.\quad
The orbits of a
standard parabolic subgroup of $\Sigma_s$ on $\{1,\ldots,s\}$ are intervals,
and
$H_g=\left\langle s_j\ \middle|\ \{j,j+1\}\subseteq I_a\text{ for some }a\right\rangle$
when every $I_a$ is an interval. Thus
\textup{(3)}$\Leftrightarrow$\textup{(4)}.
\end{proof}

\begin{proof}[Proof of Theorem~\ref{thm:intro-endomorphism-classes}]
The displayed equivalence follows from
Theorem~\ref{thm:endomorphism-classes}. The quotient identification is
Proposition~\ref{prop:Schreier}, and the orientation criterion is
Theorem~\ref{thm:orientation-descent}.
\end{proof}

The orientation criterion has the following parabolic and enumerative
consequences.

\begin{remark}\label{rem:parabolic-quotient}
Assume that every $I_a$ is an interval, and write these interval blocks
as $J_1,\ldots,J_t$ in their natural order, with $n_\alpha=|J_\alpha|$.
By Theorem~\ref{thm:orientation-descent}, $H_g$ is standard parabolic.
For $w\in\Sigma_s$, replace each value $w(j)\in J_\alpha$ by the letter
$\alpha$. The resulting word depends only on the coset $H_gw$ and gives a
bijection between $H_g\backslash\Sigma_s$ and the multiset words with
multiplicities $n_1,\ldots,n_t$. Under this bijection, the partial order
generated by the oriented nonloop edges is the usual parabolic weak order
on these multiset words (see \cite[Sections~1--2 and~9]{BjornerWachs}).
Equivalently, after loops are omitted and labels are forgotten, the oriented
quotient is its Hasse diagram. The rank function is the inversion number
of the resulting word, so the rank generating polynomial and maximum rank
are, respectively,
$$
 \frac{[s]_z!}{[n_1]_z!\cdots[n_t]_z!}
 \;\text{and}\;
 \sum_{\alpha<\beta}n_\alpha n_\beta,
$$
where $[r]_z!=\prod_{j=1}^r(1+z+\cdots+z^{j-1})$.
\end{remark}

The following example illustrates the labeled loops, parallel loops, and
descended parabolic weak order.

\begin{example}\label{ex:schreier-aaab}
Let $a,b\in\mathbb Z_{>0}$ with $a\neq b$, and consider the gap word
$(a,a,a,b)$. Its stabilizer is
$H_{(a,a,a,b)}=\operatorname{Sym}\{1,2,3\} =\langle s_1,s_2\rangle\simeq\Sigma_3$.
Figure~\ref{fig:schreier-aaab} displays the four cosets and their gap words,
with the unique entry $b$ in positions $4,3,2,1$, respectively. It records
every edge of
$\operatorname{Sch}(H_{(a,a,a,b)}\backslash\Sigma_4;s_1,s_2,s_3)$,
including six labeled loops. The two loops at each endpoint are distinct
parallel loops. The $24$ tilting parameters give four isomorphism classes
of opposite endomorphism algebras, with six parameters in each class;
every loop records a nontrivial tilting mutation that preserves the
isomorphism class of the opposite endomorphism algebra.

Since the fibers $\{1,2,3\}$ and $\{4\}$ are intervals, the weak-order
orientation descends from left to right in the figure, with ranks
$0,1,2,3$, while the tilting generation-order orientation is the reverse.
Thus the rank generating polynomial is
$\frac{[4]_z!}{[3]_z!}=1+z+z^2+z^3$.
\end{example}

\begin{figure}[!htbp]
\centering
\begin{tikzpicture}[
  vertex/.style={draw=black!65,fill=tiltingfill,rounded corners,align=center,
    minimum width=3.1cm,minimum height=9mm,inner sep=2pt,font=\scriptsize},
  elabel/.style={fill=white,inner sep=1pt,font=\small},
  every loop/.style={min distance=8mm,looseness=7}
]
\node[vertex] (stabilizerCoset) at (0,0)
  {$H_{(a,a,a,b)}$\\[-1mm]$(a,a,a,b)$};
\node[vertex] (rightSthreeCoset) at (3.7,0)
  {$H_{(a,a,a,b)}s_3$\\[-1mm]$(a,a,b,a)$};
\node[vertex] (rightSthreeStwoCoset) at (7.4,0)
  {$H_{(a,a,a,b)}s_3s_2$\\[-1mm]$(a,b,a,a)$};
\node[vertex] (rightSthreeStwoSoneCoset) at (11.1,0)
  {$H_{(a,a,a,b)}s_3s_2s_1$\\[-1mm]$(b,a,a,a)$};
\draw[semithick,draw=mutationthree] (stabilizerCoset) --
  node[elabel,text=mutationthree,above] {$3$} (rightSthreeCoset);
\draw[semithick,draw=mutationtwo] (rightSthreeCoset) --
  node[elabel,text=mutationtwo,above] {$2$} (rightSthreeStwoCoset);
\draw[semithick,draw=mutationone] (rightSthreeStwoCoset) --
  node[elabel,text=mutationone,above] {$1$} (rightSthreeStwoSoneCoset);
\path
  (stabilizerCoset) edge[loop below,semithick,draw=mutationone]
       node[elabel,text=mutationone,below] {$1$} (stabilizerCoset)
       edge[loop above,semithick,draw=mutationtwo]
       node[elabel,text=mutationtwo,above] {$2$} (stabilizerCoset)
  (rightSthreeCoset) edge[loop above,semithick,draw=mutationone]
       node[elabel,text=mutationone,above] {$1$} (rightSthreeCoset)
  (rightSthreeStwoCoset) edge[loop below,semithick,draw=mutationthree]
       node[elabel,text=mutationthree,below] {$3$} (rightSthreeStwoCoset)
  (rightSthreeStwoSoneCoset) edge[loop above,semithick,draw=mutationtwo]
       node[elabel,text=mutationtwo,above] {$2$} (rightSthreeStwoSoneCoset)
       edge[loop below,semithick,draw=mutationthree]
       node[elabel,text=mutationthree,below] {$3$} (rightSthreeStwoSoneCoset);
\end{tikzpicture}
\caption{The complete labeled Schreier multigraph
$\operatorname{Sch}
(H_{(a,a,a,b)}\backslash\Sigma_4;s_1,s_2,s_3)$ for the gap word
$(a,a,a,b)$. The horizontal edges are the nonloops; the six curved edges
are the distinct labeled loops. Edge colors agree with the numerical
mutation labels.}
\label{fig:schreier-aaab}
\end{figure}

We conclude with a remark on the closed walks arising from the stabilizer
$H_g$ and on a related coherence question for mutation-induced derived
equivalences.

\begin{remark}
A mutation word $s_{i_1}\cdots s_{i_r}=w$ determines a walk in
$\mathcal G_g$ from $T_1$ to $T_w$. This walk returns to its initial
tilting vertex if and only if $w=1$. Its image in
$H_g\backslash\mathcal G_g$ is closed if and only if
$H_gw=H_g$, or equivalently, $w\in H_g$. By
Theorem~\ref{thm:endomorphism-classes}, this is also equivalent to
$
 \End_{A_g}(T_w)^{\op}\simeq A_g.
$

These observations concern only the isomorphism classes of opposite
endomorphism algebras. They do not imply that different reduced mutation
paths with the same endpoints induce naturally isomorphic derived
equivalences. In the saturated split case, a braid-group model is available
(see \cite[Theorems~1.1 and~1.2]{Sauter}). Whether an analogous coherence
statement holds for an arbitrary positive gap word remains open.
\end{remark}

\section*{Acknowledgments}

Jiangsheng Hu was supported by the National Natural Science Foundation of
China (No.~12571035).
Xin Ma was supported by the Natural Science Foundation of Henan
(No.~262300421832) and Central Plains Science and Technology Innovation Youth
Top-notch Talent (No.~2024ZYBJRC002).
Jinbi Zhang was supported by the National Natural Science Foundation of China
(No.~12401038).
Tiwei Zhao
was supported by the National Natural Science Foundation of China (Grant
No.~12471036) and the Hubei Provincial Natural Science Foundation of China
(No.~2026AFA094).

\begin{samepage}
\medskip
\noindent
\textbf{Use of AI tools.}
After completing the initial draft of this manuscript, we used ChatGPT (OpenAI)
to assist with English-language editing, literature searches, and LaTeX revisions.
We take full responsibility for the final manuscript, including all mathematical
statements, proofs, and references.\par
\end{samepage}

\end{document}